\documentclass[3p,number,sort]{elsarticle}

\makeatletter
\def\ps@pprintTitle{%
 \let\@oddhead\@empty
 \let\@evenhead\@empty
 \def\@oddfoot{}%
 \let\@evenfoot\@oddfoot}
\makeatother

\newtheorem{theorem}{Theorem}
\newtheorem{lemma}[theorem]{Lemma}

\newdefinition{definition}[theorem]{Definition}
\newdefinition{conjecture}[theorem]{Conjecture}
\newdefinition{example}[theorem]{Example}
\newdefinition{remark}[theorem]{Remark}
\newdefinition{note}[theorem]{Note}
\newdefinition{case}[theorem]{Case}

\newproof{proof}{proof}

\usepackage{amsmath,amssymb}
\usepackage{bbm}
\usepackage{enumitem}

\usepackage{setspace}

\usepackage{hyperref}
\usepackage{cleveref}

\usepackage{float}

\usepackage{pgfplots}
\pgfplotsset{compat=newest}

\usepgfplotslibrary{patchplots}

\usepackage{tikz}
\usetikzlibrary{math}
\usetikzlibrary{external}

\allowdisplaybreaks
\usepackage[misc]{ifsym}
\newcommand{\envelope}{(\kern1pt\Letter\kern1pt)}

\newcommand{\N}{\ensuremath{\mathbb{N}}}

\newcommand{\R}{\ensuremath{\mathbb{R}}}

\newcommand{\e}{\ensuremath{\varepsilon}}

\newcommand{\Pp}{\ensuremath{\mathbb{P}}}

\newcommand{\supp}{\ensuremath{\mathrm{supp}}}

\newcommand{\dt}[2]{\mbox{#1.\hspace{.13em}#2.}}
\newcommand{\dtt}[3]{\mbox{#1.\hspace{.13em}#2.\hspace{.13em}#3.}}

\newcommand*\cvec[1]{\begin{pmatrix}#1\end{pmatrix}}

\newcommand{\sgn}{\mathrm{sgn}}
\newcommand{\Langle}{\left\langle}
\newcommand{\Rangle}{\right\rangle}

\newcommand{\sis}{\mathrm{sing}\;\mathrm{supp}}

\newcommand{\lcm}{\mathrm{lcm}}

\newcommand{\1}{\mathbbm{1}}

\renewcommand{\(}{\left(}
\renewcommand{\)}{\right)}
\renewcommand{\[}{\left[}
\renewcommand{\]}{\right]}
\renewcommand{\S}{\mathcal{S}}
\renewcommand{\O}{\mathcal{O}}

\begin{document}

\begin{frontmatter}

\title{Higher order analysis of the geometry of singularities using the Taylorlet transform}

\author{Fink, Thomas}
\ead{thomas.fink@uni-passau.de}

\address{Universit\"at Passau}

\begin{abstract}
We consider an extension of the continuous shearlet transform which additionally uses higher order shears. This extension, called the Taylorlet transform, allows for a detection of the position, the orientation, the curvature and other higher order geometric information of singularities. Employing the novel vanishing moment conditions of higher order, $\int_\R g(t^k)t^m dt=0$, on the analyzing function, we can show that the Taylorlet transform exhibits different decay rates for decreasing scales depending on the choice of the higher order shearing variables. This enables a more robust detection of the geometric information of singularities. Furthermore, we present a construction that yields analyzing functions which fulfill vanishing moment conditions of different orders simultaneously.
\end{abstract}

\end{frontmatter}

\section{Introduction}

Edges are crucial features in image processing as they contain a considerable part of an image's visual information. Hence, their detection and analysis play a major role in computer vision and are vital in various fields of application including object recognition, image enhancement and feature extraction. One of the most important characteristics of an edge is its curvature. It is a distinctive size for contours and hence especially valuable for shape recognition \cite{PKM97,MEO11}. Furthermore, the edge curvature plays a vital role in the human visual perception - especially corners and edge points with high curvature are crucial for the eye's recognition of shapes \cite{AtAr56,blov74}.

A plethora of different methods exists for the task of edge detection. Multi-scale approaches play a special role as they offer a good noise robustness and are motivated from a continuous setting where the term of an edge finds a more general mathematical analogue in the singular support. The detection of this feature \dt ie, the distinction between regular and singular points by the decay rate of continuous multi-scale methods such as the continuous wavelet transform has been thoroughly discussed in the literature, \dt eg, for the continuous wavelet transform in \cite{mahw92}. In addition to the identification of singularities, the continuous curvelet and shearlet transform as well offer a detection of directional information, \dt ie, a resolution of the wavefront set \cite{CaDo05a,KuLa09}. When the latter is additionally endowed with second order shears
$$S_s:\R^2\to\R^2,\cvec{x_1 \\ x_2}\mapsto \cvec{x_1+s_1 x_1+s_2 x_2^2 \\ x_2},$$
the resulting bendlet transform is capable of extracting the curvature of an edge \cite{LePeSch16}. 

In this paper, we introduce the Taylorlet transform, which utilizes higher order shears like the bendlet transform, and allows for an extraction of the position, orientation, curvature and higher order geometric information of edges. It extends the bendlet transform by conditions that ensure a high decay rate for a more robust detection of the desired features. The approach is based on a modeling of a singular support as graph of a singularity function $q\in C^\infty(\R)$, \dt ie, $\sis(f)=\{x\in\R^2:\ x_1=q(x_2)\}$. The geometrical data accessible by the Taylorlet transform consists of the Taylor coefficients of $q$. In this perspective, the 2D continuous wavelet and shearlet transform essentially identify the $0^\mathrm{th}$ rsp. the $0^\mathrm{th}$ and $1^\mathrm{st}$ Taylor coefficients of the singularity function. 

For these detection properties, vanishing moment conditions for the respective analyzing function are essential. They are responsible for the ability of the continuous wavelet transform to detect singularities of high regularity \cite[Thm 3]{mahw92} and ensure the decay rate of the continuous shearlet transform for decreasing scales \cite[Thm 3.1]{gr11}. We will show in this paper that analogously
an analyzing Taylorlet $\tau(x)=g(x_1)\cdot h(x_2)$ fulfilling vanishing moment conditions of the type
$$\int_\R g(x_1^k)x_1^m dx_1 = 0\in\R\text{ for all }k\in\{1,\ldots,n\},$$
are of similar importance for the extraction of higher order geometric information.

The paper is organized as follows. In section 2 the Taylorlet transform and all basic definitions are introduced, before we describe the construction of analyzing Taylorlets of arbitrary order in section 3. In section 4 the main result is stated and explained in detail and afterwards proved in section 5. Section 6 shows some numerical examples of the Taylorlet transform. Finally, in section 7 we discuss open problems and possible extensions of the Taylorlet transform.

\section{Basic definitions and notation}

The goal of the Taylorlet transform is a precise analytical description of the singular support of the analyzed function $f$. To this end, we assume that we can represent $\sis(f)$ locally as the graph of a singularity function $q\in C^\infty(\R)$ and describe $\sis(f)$ by the Taylor coefficients of $q$. These coefficients can be found by observing the decay rate of the Taylorlet transform. In this way the continuous shearlet transform essentially delivers a local linear approximation to the singular support which can be regarded as a first order Taylor polynomial of the singularity function $q$. Hence, we will use it as a starting point for the construction of the Taylorlet transform. To this end, we need an extension of the classical shear: we will use a modification of the higher order shearing operators introduced in \cite{LePeSch16}.

\begin{definition}\label{operators}
For $n\in\N$ and $s=(s_0,\ldots,s_n)^T\in\R^{n+1}$ the $n$-th order shearing operator is defined as
$$S_s^{(n)}:\R^2\to\R^2,\quad S_s^{(n)}(x):=\cvec{x_1 + \sum_{\ell=0}^n \frac{s_\ell}{\ell !}\cdot  x_2^\ell \\ x_2}.$$
In contrast to \cite{LePeSch16}, the higher order shearing operator here also includes a simple translation along the $x_1$-axis in the form of $s_0$. This is included to emphasize the Taylor coefficient perspective on the singular support of the analyzed function.
\end{definition}
Furthermore, for $a,\alpha>0$ and we use an $\alpha$-scaling matrix \cite{LePeSch16}
$$A_a^{(\alpha)}:=\begin{pmatrix}
a & 0 \\
0 & a^\alpha
\end{pmatrix}.$$

\begin{definition}[Iterated integrals]
Let $k\in\N$, $u\in\R$ and $\phi\in L^1(\R)$. We define $I_+ \phi(u) := \int_{-\infty}^u \phi(v)\ dv$, $I_-\phi(u) := \int_u^\infty \phi(v)dv$ and $I_\pm^{k+1}\phi = I_\pm\circ I_\pm^k\phi$.
\end{definition}

A central property of analyzing functions of continuous multi-scale transforms is the vanishing moment condition which plays a crucial role for wavelets in order to detect singularities of a certain smoothness \cite{mahw92}. For shearlets there exist analogous results \cite{gr11}. Pursuing a similar goal, the following definition of analyzing Taylorlets incorporates some special vanishing moment properties.

\begin{definition}[Vanishing moments of higher order, analyzing Taylorlet, restrictiveness]
We say that a function $f:\R\to\R$ has $r$ vanishing moments of order $n$ if
$$\int_\R f\big(\pm t^k\big)t^m dt = 0$$
for all $m\in\{0, \ldots, kr-1\}$ and for all $k\in\{1,\ldots,n\}.$

Let $g,h\in\S(\R)$ such that $g$ has $r$ vanishing moments of order $n$. We call the function
$$\tau=g\otimes h$$
an analyzing Taylorlet of order $n$ with $r$ vanishing moments.

We say $\tau$ is restrictive, if additionally 
\begin{enumerate}[label=(\roman*)]
\item $I_+^j g(0)\ne 0 \quad \text{for all }j\in\{0,\ldots,r\}$ and
\item $\int_\R h(t)dt\ne 0$.
\end{enumerate}
\end{definition}

In \cref{construction} we show that for arbitrary $n\in\N$, the set of restrictive analyzing Taylorlets of order $n$ is not empty. 

We define the Taylorlet transform as follows.

\begin{definition}[Taylorlet transform]
Let $\tau\in\S(\R^2)$ be an analyzing Taylorlet. Let $n\in\N$, $\alpha>0$, $t\in\R$, $a>0$ and $s\in\R^{n+1}$. We define
$$\tau^{(n,\alpha)}_{a,s,t}(x):= a^{-(1+\alpha)/2}\cdot \tau\(A_{\frac 1a}^{(\alpha)} S_{-s}^{(n)}\cvec{x_1 \\ x_2-t}\)\quad\text{for all }x=(x_1,x_{2})\in\R^{2}.$$
Whenever the values of $\alpha$ and $n$ are clear, we will omit these indices and write $\tau_{a,s,t}$ instead. \\
The Taylorlet transform \dtt wrt $\tau$ of a tempered distribution $f\in \S'(\R^2)$ is defined as
\begin{align*}
\mathcal{T}^{(n,\alpha)}:\S'\(\R^2\)\to C^\infty\(\R_+\times\R^{n+1}\times\R\),\quad \mathcal{T}^{(n,\alpha)}f(a,s,t) = \Langle f, \tau^{(n,\alpha)}_{a,s,t}\Rangle.
\end{align*}
\end{definition}

The smoothness of the Taylorlet transform results from $\tau\in\S\(\R^2\)$ and from the smoothness of the translation, the higher-order shear and the dilation.

We want to distinguish between the right choice of shearing parameters $s=(s_0,\ldots,s_n)$ and the incorrect ones by comparing the respective decay rates of the Taylorlet transform. We will show that, if the choice is incorrect, the Taylorlet transform decays fast because of the vanishing moments of higher order. The restrictiveness on the other hand makes sure that the Taylorlet transform decays slowly at a correct choice of shearing parameters.

\section{Construction of a Taylorlet}

In the setting of the continuous shearlet transform the vanishing moment property \dtt wrt the $x_1-$direction is inherently given by the definition of a classical shearlet $\psi$ \cite{KuLa09}
$$\hat\psi(\xi) = \hat\psi_1(\xi_1)\cdot \hat\psi_2\(\frac{\xi_2}{\xi_1}\),$$
since $0\notin\supp\big(\hat\psi_1\big)$. In the Taylorlet setting, the function $g$ essentially takes over the role of $\psi_1$. Unfortunately, vanishing moments of $g$ alone are not sufficient for the construction of an analyzing Taylorlet, since we additionally need vanishing moments of $g(\pm t^k)$ for all $k\le n$. This complicates the situation as we cannot rely on a construction via the Fourier transform.

We will first present the construction procedure and exemplary images of each construction step starting from the function $\phi(t)=e^{-t^{2}}$. Later, we prove that the resulting function is a restrictive Taylorlet of order $n$ with $r$ vanishing moments. \\[3mm]

{\bf General setup}\\
\begin{enumerate}[label=\Roman*.] 
\begin{minipage}{.65\textwidth}
\item We start with an even function $\phi\in\S(\R)$ fulfilling
$$\phi^{(k)}(0)\ne 0\quad \Leftrightarrow \quad k\ \mathrm{ mod }\ 2=0.$$
This condition is necessary for the Taylorlet to be a Schwartz function. For instance, we can choose $\phi(t)=e^{-t^2}$.
\end{minipage}
\hfill
\begin{minipage}{.25\textwidth}
\begin{tikzpicture}[scale=.5]
\begin{axis}[samples=500,domain=-6:6]
\addplot[ultra thick,blue!80!black]plot (\x, {exp(-pow(\x,2))});
\addlegendentry{$\phi(t)=e^{-t^2}$};
\end{axis}
\end{tikzpicture}
\end{minipage}
\vspace{4mm}

\begin{minipage}{.65\textwidth}
\item Let $v_n\in\N$ be the least common multiple of the numbers $1,\ldots, n$. We define $$\phi_n(t):=\phi\big(t^{v_n}\big)\quad\text{for all }t\in\R.$$ This function is still in $\S(\R)$ and fulfills
$$\phi_n^{(k)}(0)\ne 0\quad \Leftrightarrow \quad k\ \mathrm{ mod }\ 2v_n=0.$$
\end{minipage}
\hfill
\begin{minipage}{.25\textwidth}
\begin{tikzpicture}[scale=.5]
\begin{axis}[samples=500,domain=-6:6]
\addplot[ultra thick,blue!80!black]plot (\x, {exp(-pow(\x,4))});
\addlegendentry{$\phi_2(t)$};
\end{axis}
\end{tikzpicture}
\end{minipage}
\vspace{4mm}

\begin{minipage}{.65\textwidth}
\item We define $$\phi_{n,r}:=\tfrac 1{(2rv_n)!}\cdot\phi_n^{(2r v_n)}.$$ This function has $2rv_n$ vanishing moments since each derivative generates one further vanishing moment. Furthermore, the function is also in $\S(\R)$ and fulfills
$$\phi_{n,r}^{(k)}(0)\ne 0\quad \Leftrightarrow \quad k\ \mathrm{ mod }\ 2v_n=0.$$
\end{minipage}
\hfill
\begin{minipage}{.25\textwidth}
\begin{tikzpicture}[scale=.5]
\begin{axis}[samples=500,domain=-6:6]
\addplot[ultra thick,blue!80!black]plot (\x, {1/630*exp(-pow(\x,4))*(315 - 51660*pow(\x,4) + 286020*pow(\x,8) - 349440*pow(\x,12) + 142464*pow(\x,16) - 21504*pow(\x,20) + 1024*pow(\x,24))});
\addlegendentry{$\phi_{2,2}(t)$};
\end{axis}
\end{tikzpicture}
\end{minipage}
\vspace{4mm}

\begin{minipage}{.65\textwidth}
\item We define $$\tilde\phi_{n,r}=\phi_{n,r}\(|t|^{\frac 1{v_n}}\)\quad\text{for all }t\in\R.$$ The concatenation with the function $|\cdot|^{\frac 1{v_n}}$ ensures that $\tilde\phi_{n,r}$ has vanishing moments of order $n$. The function $\tilde\phi_{n,r}$ fulfills
$$\tilde\phi_{n,r}^{(k)}(0)\ne 0\quad \Leftrightarrow \quad k\ \mathrm{ mod }\ 2=0.$$
So, it is smooth despite the singularity of $|\cdot|^{1/v_n}$ and in $\S(\R)$, as well.
\end{minipage}
\hfill
\begin{minipage}{.25\textwidth}
\begin{tikzpicture}[scale=.5]
\begin{axis}[samples=500,domain=-6:6]
\addplot[ultra thick,blue!80!black]plot (\x, {1/630*exp(-pow(\x,2))*(315 - 51660*pow(\x,2) + 286020*pow(\x,4) - 349440*pow(\x,6) + 142464*pow(\x,8) - 21504*pow(\x,10) + 1024*pow(\x,12))});
\addlegendentry{$\tilde\phi_{2,2}(t)$};
\end{axis}
\end{tikzpicture}
\end{minipage}
\vspace{4mm}

\begin{minipage}{.65\textwidth}
\item For all $t\in\R$, we define $$g(t):=(1+t)\tilde\phi_{n,r}(t).$$ This step guarantees that the necessary properties of $g$ for the restrictiveness are fulfilled. Furthermore, $g\in\S(\R)$.
\end{minipage}
\hfill
\begin{minipage}{.25\textwidth}
\begin{tikzpicture}[scale=.5]
\begin{axis}[samples=500,domain=-6:6]
\addplot[ultra thick,blue!80!black]plot (\x, {1/630*exp(-pow(\x,2))*(1+\x)*(315 - 51660*pow(\x,2) + 286020*pow(\x,4) - 349440*pow(\x,6) + 142464*pow(\x,8) - 21504*pow(\x,10) + 1024*pow(\x,12))});
\addlegendentry{$g(t)$};
\end{axis}
\end{tikzpicture}
\end{minipage}
\vspace{4mm}

\begin{minipage}{.65\textwidth}
\item We choose a function $h\in\S(\R)$ such that $\int_\R h(t)dt\ne 0$ and define the Taylorlet $\tau:= g\otimes h.$ Since $g,h\in\S(\R)$, we have $\tau\in \S(\R^2)$.
\end{minipage}
\end{enumerate}

\begin{figure}[h]
\centering
\includegraphics[trim={5cm 0 3cm 0},clip,angle=90,width=.3\textwidth]{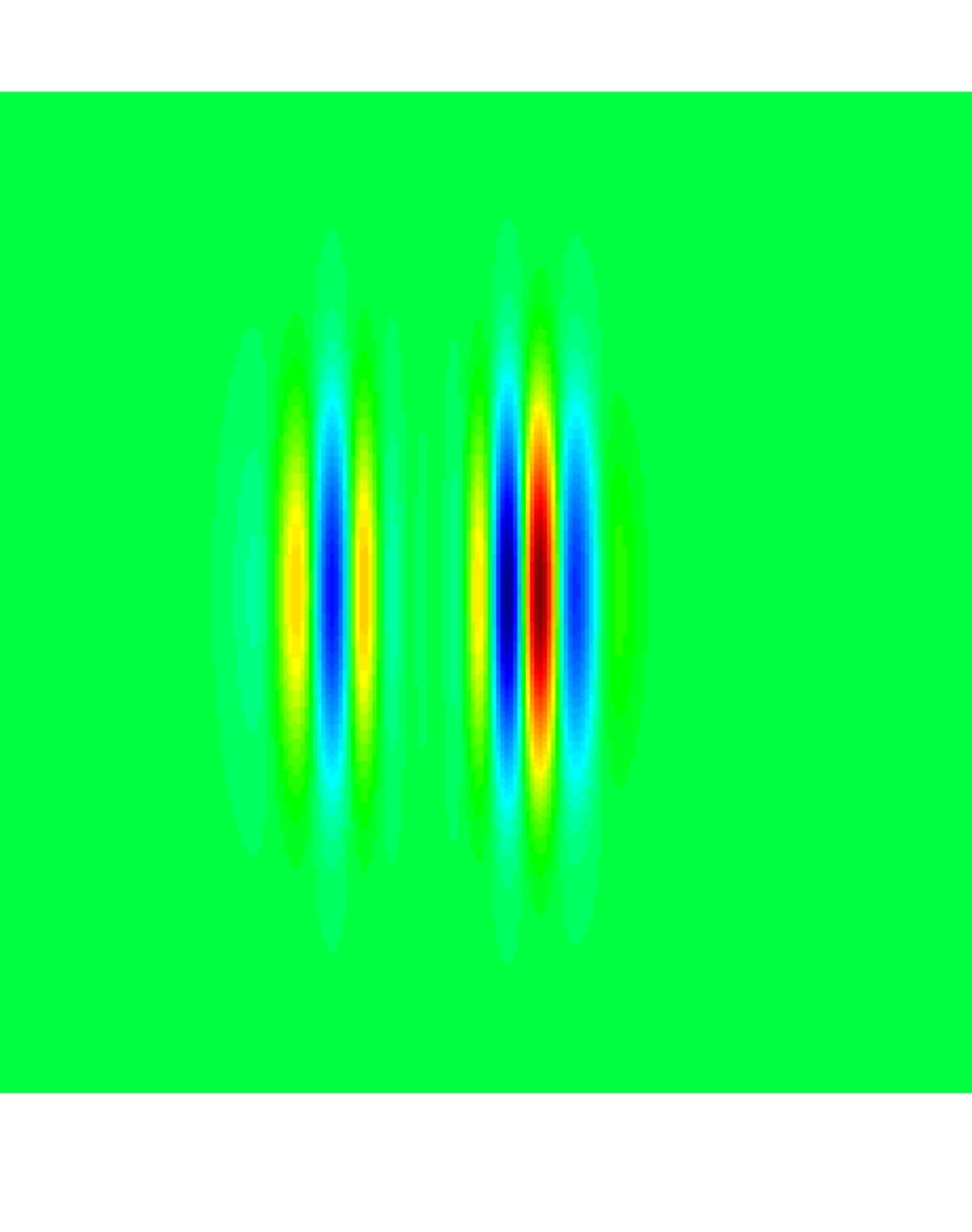}
\includegraphics[trim={5cm 0 3cm 0},clip,angle=90,width=.3\textwidth]{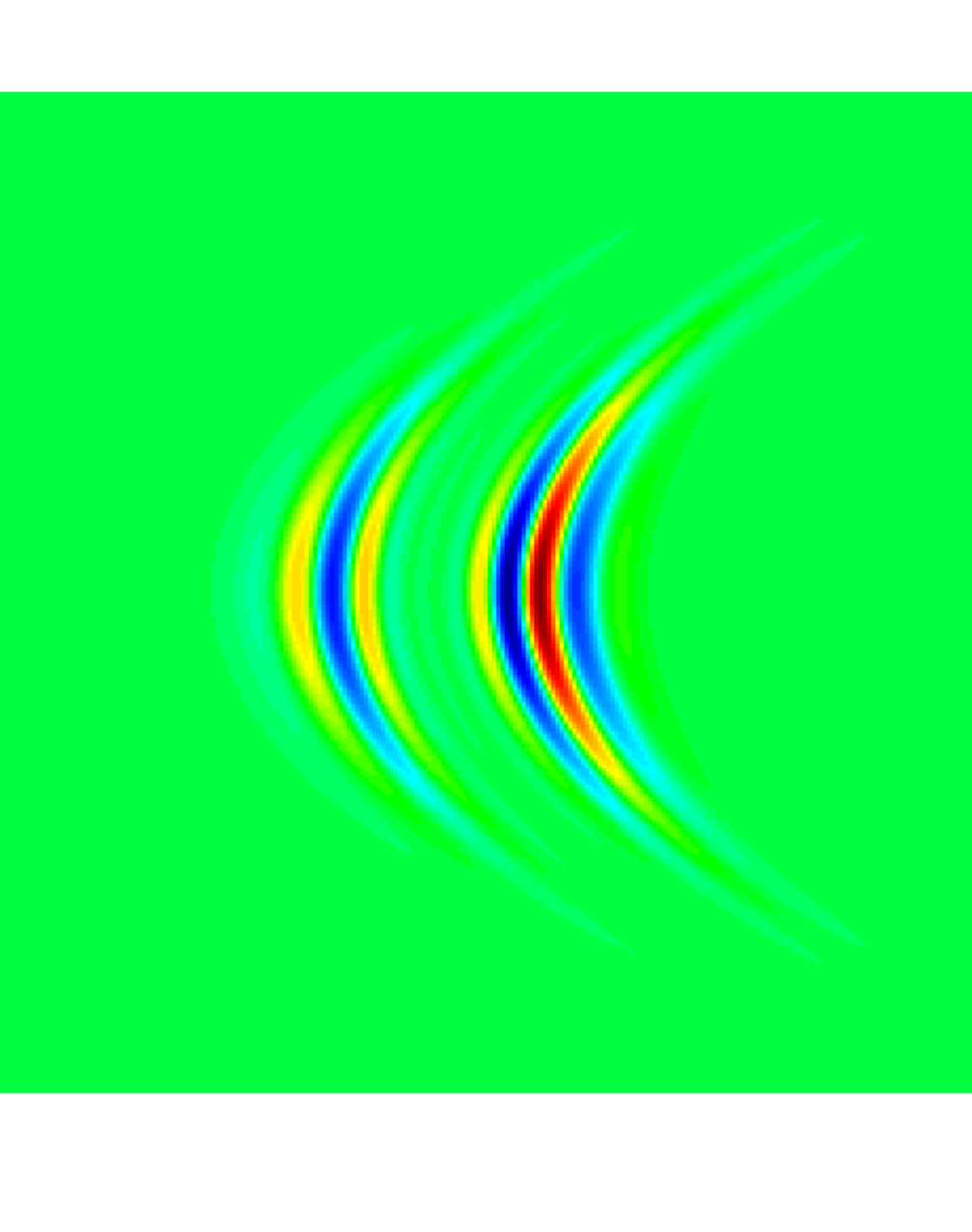}
\includegraphics[trim={5cm 0 3cm 0},clip,angle=90,width=.3\textwidth]{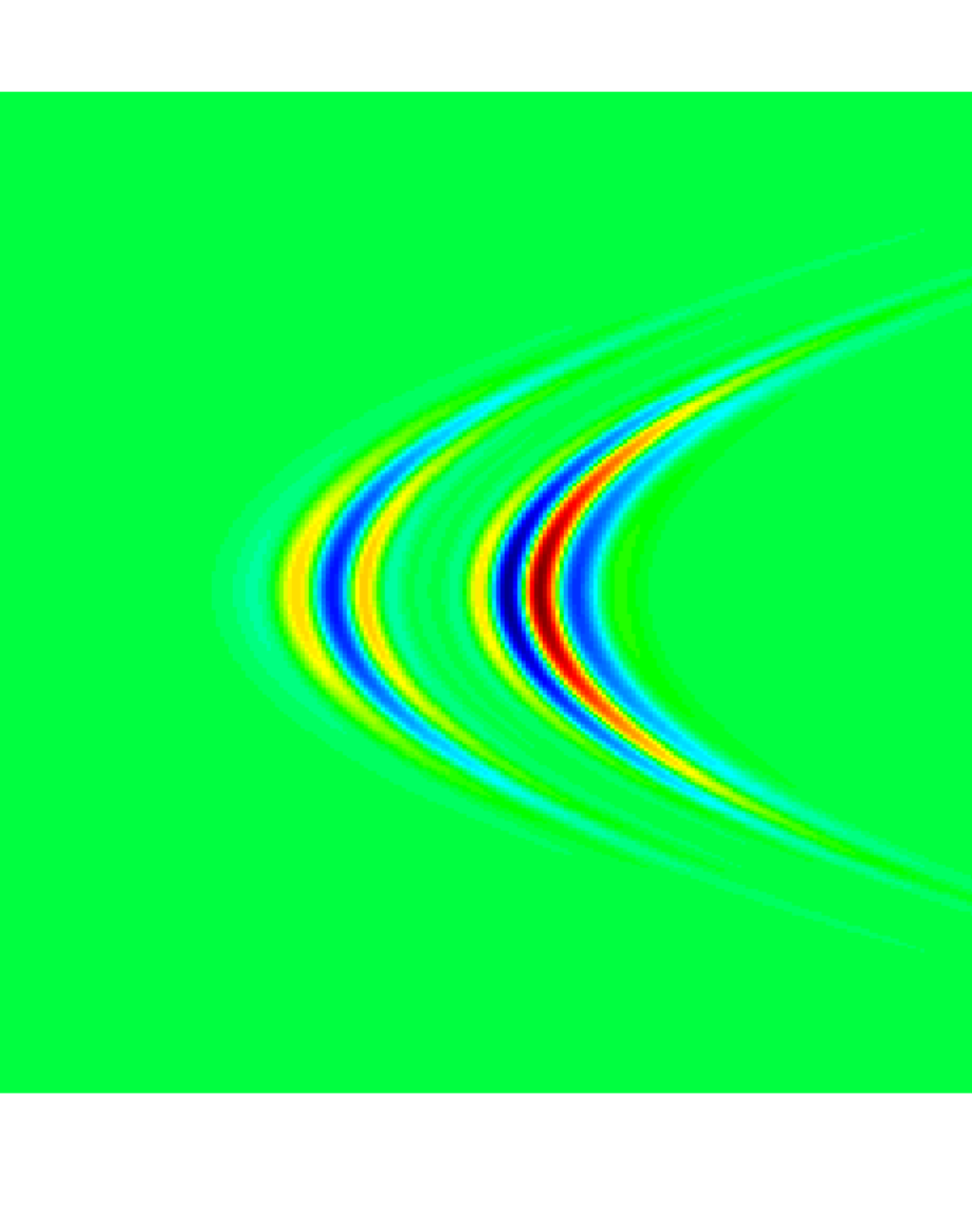}
\caption{Plots of the example Taylorlet, starting from the function $\phi(t):=e^{-t^{2}}$, $g(t):=(1+t)\cdot\tilde\phi_{2,2}(t)$, $h(t):=e^{-t^2}$, given in (\ref{example}). Increasing curvature from left to right with $s_2=0$ (left), $s_2=1$ (center), $s_2=2$ (right).}
\label{Taylorfig}
\end{figure}

In the following theorem we will prove some properties of the function $\tau$ generated by the steps I-VI above.

\begin{theorem}\label{construction}
Let $r,n\in\N$. The function $\tau$ described in the general setup exhibits the following properties:
\begin{enumerate}[label=\roman*)]
\item $\tau\in \S(\R^2)$.
\item $\tau$ is an analyzing Taylorlet of order $n$ with $2r-1$ vanishing moments.
\item $\tau$ is restrictive.
\end{enumerate}

\end{theorem}

\begin{proof}

~\\[3mm]
{\it i)} Since $\tau= g\otimes h$ and due to the general setup $h\in\S(\R)$, we only need to prove that $g\in\S(\R)$. To this end, we show that the Schwartz properties are consecutively passed on to the next function through every step of the general setup. 

First, we observe that with the change of the argument $\phi_n:=\phi\big((\cdot)^{v_n}\big)$ in step II the Schwartz properties of $\phi$ remain. Furthermore, we obtain with the condition from step I that 
$$\phi_n^{(k)}(0)\ne 0\quad \Leftrightarrow \quad k\bmod 2v_n=0.$$
This property is invariant under the action of step III, \dt ie, 
$$0\ne\phi_{n,r}^{(k)}(0)=\phi_n^{(2rv_n+k)}(0)\quad \Leftrightarrow \quad k\bmod 2v_n=0$$
 for all $r\in\N$. After step IV the function $\tilde\phi_{n,r}:=\phi_{n,r}\(|\cdot|^{\frac 1{v_n}}\)$ is clearly smooth on every set not containing the origin. In order to show the smoothness of $\tilde\phi_{n,r}$ in the origin, we use Taylor's theorem to approximate $\phi_{n,r}$ by a Taylor polynomial. We thus obtain that
$$\phi_{n,r}(t)=\sum_{k=0}^K \frac{\phi_{n,r}^{(2kv_n)}(0)}{(2kv_n)!}\cdot t^{2kv_n}+o\(t^{2Kv_n}\)\quad\text{for }t\to 0.$$
Hence, we can approximate $\tilde\phi_{n,r}$ by a sequence of polynomials, as well.
$$\tilde\phi_{n,r}(t)=\phi_{n,r}\(|t|^{\frac 1{v_n}}\)=\sum_{k=0}^K \frac{\phi_{n,r}^{(2kv_n)}(0)}{(2kv_n)!}\cdot t^{2k}+o\(t^{2K}\)\quad\text{for }t\to 0.$$
Consequently, $\tilde\phi_{n,r}$ is smooth and inherits the Schwartzian decay property of $\phi_{n,r}$. Hence, $\tilde\phi_{n,r}\in\S(\R)$. In the last step we get that
$$g(t):=(1+t)\cdot \tilde\phi_{n,r}(t)$$
is a Schwartz function.

~\\[3mm]
{\it ii)} We will prove this statement in three steps. First, we will show that $\phi_{n,r}$ has $2rv_n$ vanishing moments. In a second step we will prove that $\tilde\phi_{n,r}$ has $2r$ vanishing moments of order $n$ and in the last part, we show that $g$ has $2r-1$ vanishing moments of order $n$. \\[3mm]
{\sc Step 1}

As shown in the proof of $i)$, $\phi_n,\phi_{n,r}\in\S(\R)$. Hence, their Fourier transforms exist and we obtain
$$\widehat{\phi_{n,r}}(\omega)=\(\phi_n^{(2rv_n)}\)^\land(\omega)=(-1)^{rv_n}\omega^{2rv_n}\widehat{\phi_n}(\omega).$$
Consequently, $\widehat{\phi_{n,r}}$ has a root of order at least $2rv_n$ in the origin and hence $\phi_{n,r}$ has at least $2rv_n$ vanishing moments. \\[3mm]
{\sc Step 2}

We now prove that $\tilde\phi_{n,r}:=\phi_{n,r}\(|\cdot|^{1/v_n}\)$ has $r$ vanishing moments of order $n$. Now let $k\in\{1,\ldots,n\}$. Then
\begin{align*}
\int_{\R} \tilde\phi_{n,r}\big(t^k\big) t^m dt &= \int_{\R} \phi_{n,r}\big(|t|^{k/v_n}\big) t^m dt \\
&= \int_{\R} \phi_{n,r}(|u|) u^{m v_n/k}\cdot\frac{v_n}{k} u^{v_n/k-1}du \\
&= \frac{v_n}k\cdot \int_{\R} \phi_{n,r}(u) u^{(m+1) v_n/k-1} du.
\end{align*}
Since $v_n$ is the least common multiple, $v_n/k\in\N$ and thus the upper expression vanishes for all $m< 2k\cdot r$. \dt ie, $\tilde\phi_{n,r}$ has $2r$ vanishing moments of order $n$. \\[3mm]
{\sc Step 3}

Now we will show that $g$ has $2r-1$ vanishing moments of order $n$.
\begin{align*}
\int_\R g\(t^k\)t^m dt &= \int_\R (1+t^k)\tilde\phi_{n,r}(t^k)t^m dt \\
&= \int_\R \tilde\phi_{n,r}(t^k)t^m dt + \int_\R \tilde\phi_{n,r}(t^k)t^{k+m} dt.
\end{align*}
Due to the result of {\sc step 2}, this expression vanishes if $m+k<2k\cdot r$. Hence, $g$ has $2r-1$ vanishing moments of order $n$ and $\tau$ is an analyzing Taylorlet of order $n$ with $2r-1$ vanishing moments. \\[3mm]
{\it iii)} In order to prove that $\tau= g\otimes h$ is restrictive, it is sufficient to show that 
$$I^j_+ g(0)\ne 0\quad\text{ for all }j\in\{0,\ldots,2r-1\}$$
since $\int_\R h(t)dt\ne 0$ is already given in step VI of the general setup. This property will be shown in two steps. First we will prove the sufficiency of
$$\int_0^\infty\tilde\phi_{n,r}(t)t^{2m+1}dt\ne 0$$
for all $m\in\{0,\ldots,r-1\}$. Afterwards we will reduce this property to the already proven property that 
$$\phi_{n,r}^{(k)}(0)\ne 0\quad \Leftrightarrow \quad k\bmod 2v_n=0.$$
{\sc Step 1}

\begin{align*}
I_-^j g(0) &= \int_0^\infty g(t)t^{j-1}dt \\
&= \int_0^\infty \tilde\phi_{n,r}(t)(1+t)t^{j-1} dt \\
&= \int_0^\infty \tilde\phi_{n,r}(t)t^{j-1} dt + \int_0^\infty \tilde\phi_{n,r}(t)t^{j} dt.
\end{align*}
Since $\tilde\phi_{n,r}$ is an even function with $2r$ vanishing moments, we obtain for $k<r$ that
$$\int_0^\infty \tilde\phi_{n,r}(t) t^{2k} dt = \frac 12 \int_\R \tilde\phi_{n,r}(t) t^{2k} dt = 0.$$
Hence, we can conclude for the iterated integral of $g$ that
$$I_-^j g(0) = \begin{cases}
\int_0^\infty \tilde\phi_{n,r}(t) t^{j-1} dt, &\text{if }j\bmod 2 = 0, \\
\int_0^\infty \tilde\phi_{n,r}(t) t^{j} dt, &\text{if }j\bmod 2 = 1.
 \end{cases}$$
Since $g$ has $2r-1$ vanishing moments, the statements $I_+ ^j g(0)\ne 0$ and $I_-^j g(0)\ne 0$ are equivalent. Consequently, we obtain that $I_+^j g(0)\ne 0$ for all $j\in\{0,\ldots,2r-1\}$ is equivalent to 
$$\int_0^\infty \tilde\phi_{n,r}(t)t^{2m+1}dt\ne 0\quad\text{for all }m\in\left\{0,\ldots,r-1\right\}.$$
{\sc Step 2}

\begin{align*}
\int_0^\infty\tilde\phi_{n,r}(t)t^{2m+1}dt &\stackrel{\phantom{part. int.}}{=} \int_0^\infty \phi_{n,r}\(t^{\frac 1{v_n}}\)t^{2m+1}dt \\
&\hspace*{1.45mm} \stackrel{t=u^{v_n}}{=} \ v_n\cdot \int_0^\infty \phi_{n,r}(u)u^{v_n(2m+1)}\cdot u^{v_n-1}du  \\
&\stackrel{\phantom{part. int.}}{=} v_n\cdot \int_0^\infty \phi_{n,r}(u)u^{2v_n(m+1)-1}du  \\
&\stackrel{\phantom{part. int.}}{=} v_n\cdot \int_0^\infty \phi_{n}^{(2rv_n)}(u)u^{2v_n(m+1)-1}du \\
&\stackrel{\text{part. int.}}{=} [2v_n(m+1)-1]!\cdot v_n\cdot \int_0^\infty \phi_n^{(2v_n[r-m-1]+1)}(u)du \\
&\stackrel{\phantom{part. int.}}{=} -[2v_n(m+1)-1]!\cdot v_n\cdot \phi_n^{(2v_n[r-m-1])}(0).
\end{align*}
The last expression does not vanish for any $m\in\{0,\ldots,r-1\}$ because $\phi_n^{k}(0)\ne 0\Leftrightarrow k\bmod 2v_n=0$. Hence,
$$\int_0^\infty\tilde\phi_{n,r}(t)t^{2m+1}dt \ne 0\quad\text{for all }m\in\{0,\ldots,r-1\}.$$
Due to {\sc step 1} we can conclude that $I_+^j g(0)\ne 0$ for all $j\in\{0,\ldots,r-1\}.$
\flushright{$\Box$} \end{proof}

\begin{remark}

The sequence $v_n=\lcm\{1,\ldots,n\}$ is innately connected to the second Chebyshev function which plays a crucial role for the prime number theorem. The second Chebyshev function is defined as
$$\psi(x)=\sum_{p\in\Pp,k\in\N:\ p^k \le x}\log p.$$
Its relation to the sequence $v_n$ is given by the equation $v_n=e^{\psi(n)}$. The second Chebyshev function $\psi$ itself is connected to the prime number theorem. It states that the prime counting function $\pi(x)=|\{p\in\Pp:\ p\le x\}|$ exhibits the asymptotics 
$$\lim_{x\to\infty}\frac{\pi(x)\log x}x = 1.$$
This statement can be proven via a relation to $\psi$, since it can be shown \cite[Theorem 4.4]{Ap76} that it is equivalent to $\psi$ having the asymptotic behavior
\begin{equation}\label{psi} \lim_{x\to\infty}\frac{\psi(x)}x = 1.\end{equation}
This property is easier to show and can be proven by a relation to the Riemann Zeta function. On the other hand, (\ref{psi}) also provides the asymptotics for $v_n$:
$$\lim_{n\to\infty}\frac{\log v_n} n=1\text{ for }n\to\infty.$$

\end{remark}

\section{Main result}

We introduce the class of functions that we consider in our main result.

\begin{definition}[Feasible function, singularity function]
Let $j\in\N$, $q\in C^\infty(\R)$ and let
$$f(x):=I^j_\pm\delta(x_1-q(x_2)).$$
Then $f$ is called a $j$-feasible function with singularity function $q$.
\end{definition}

The variable $j$ describes the smoothness of $f$. In terms of Sobolev spaces we obtain for $j\ge 1$, that $f\in W^{j-1,\infty}(\R^2)$. For instance, by choosing $q(x_2)= x_2^2$ for all $x_2\in\R$ and $j\ge 1$, we obtain the function
$$f(x)= \frac {(x_1-x_2^2)^{j-1}}{(j-1)!} \cdot H\(\pm \big(x_1- x_2^2\big)\),$$
where $H:\R\to\R, t\mapsto \1_{\R^+}(t)$ is the Heaviside step function.

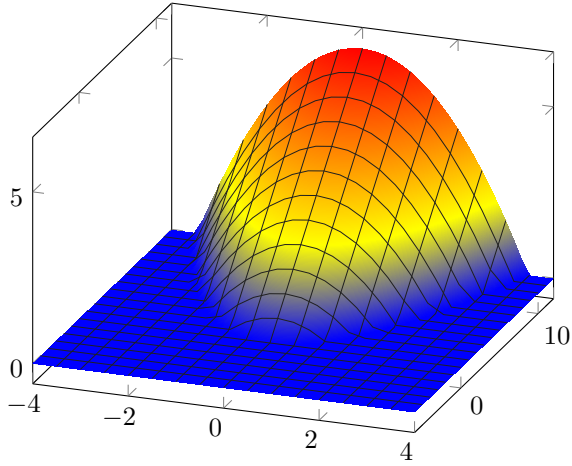
\begin{figure}[H]
\begin{center}
\begin{tikzpicture}
	\begin{axis}[view/h=20]
	\addplot3[
		surf,		
		shader=interp,
		samples=50,
		domain=-4:4,y domain=-6:12] 
		{( abs( y - x^2 ) + ( y - x^2 ) ) /4};
		\foreach \xx in {0,.5,1,1.5,2,2.5,3,3.5}
      {
        \addplot3[domain=-6:12, line width=0.05mm, mark=none, color=black!75!gray, samples y=0]
         ({\xx}, {x},{( abs( x - \xx^2 ) + ( x - \xx^2 ) )/4});
      }
      
      		\foreach \xx in {.5,1,1.5,2,2.5,3,3.5}
      {
        \addplot3[domain=-6:12, line width=0.05mm, mark=none, color=black!75!gray, samples y=0]
         ({-\xx}, {x},{( abs( x - \xx^2 ) + ( x - \xx^2 ) )/4});
      }

      \foreach \yy in {-5,-4,-3,-2,-1,0,1,2,3,4,5,6,7,8,9,10,11}
      {
         \addplot3[domain=-4:4, line width=0.05mm, mark=none, color=black!75!gray, samples y=0]
         ({x}, {\yy}, {( abs( \yy - x^2 ) + ( \yy - x^2 ) )/4});
      }
	\end{axis}
\end{tikzpicture}
\caption{Plot of the 2-feasible function $f(x)= (x_1-x_2^2)\cdot H(x_1-x_2^2)$}
\end{center}
\end{figure}
In order to classify the shearing variables \dtt wrt the local properties of the singularity function $q$, we introduce the concept of the highest approximation order.

\begin{definition}[Highest approximation order]
Let $j,n\in\N$ and let $f$ be a $j-$feasible function with the singularity function $q$. Furthermore, let $\alpha>0$, $t\in\R$, $a>0$ and $k\in\{0,\ldots,n-1\}$. If $s_\ell=q^{(\ell)}(t)$ for all $\ell\in\{0,\ldots,k\}$ and $s_{k+1}\ne q^{(k+1)}(t)$, we say that $k$ is the highest approximation order of the shearing variable $s=(s_0,\ldots,s_n)$ for $f$ in $t$.
\end{definition}

The following theorem states the main result of this article and treats the classification of the Taylorlet transform's decay \dtt wrt the highest approximation order.

\begin{theorem} \label{main}
Let $r,n\in\N$ and let $\tau$ be an analyzing Taylorlet of order $n$ with $r$ vanishing moments. Let furthermore $j< r$, $t\in\R$ and let $f$ be a $j-$feasible function.
\begin{enumerate}

\item Let $\alpha>0$. If $s_0\ne q(t)$, the Taylorlet transform has a decay of
$$\mathcal{T}^{(n,\alpha)}f(a,s,t)=\O\big(a^N\big)\quad\text{for }a\to 0$$
for all $N>0$.

\item Let $\alpha<\frac 1n$ and let $k\in\{0,\ldots,n-1\}$ be the highest approximation order of $s$ for $f$ in $t$. Then the Taylorlet transform has the decay property
$$\mathcal{T}^{(n,\alpha)}f(a,s,t)=\O\(a^{j+(\alpha-1)/2+(r-j)[1-(k+1)\alpha]}\)\quad\text{for }a\to 0.$$

\item Let $\alpha>\frac 1{n+1}$ and let $\tau$ be restrictive. If $n$ is the highest approximation order of $s$ for $f$ in $t$, then the Taylorlet transform has the decay property
$$\mathcal{T}^{(n,\alpha)}f(a,s,t)\sim a^{j+(\alpha-1)/2}\quad\text{for }a\to 0.$$
\end{enumerate}

\end{theorem}

\begin{remark}

At this point we want to highlight the importance of the restrictiveness for the Taylorlet transform. This property makes sure that a Taylorlet transform of order $n$ decays slowly if the highest approximation order is $n$. If a Taylorlet lacks the restrictiveness, we can construct an example function whose Taylorlet transform is equal to zero for all $a>0$ if the highest approximation order is $n$.

Let $\tau=g\otimes h$ be a Taylorlet of order $n$ with $r$ vanishing moments such that there is a $k\in\N$ with 
$$\int_0^\infty g(t)t^k dt = 0.$$ 
Furthermore let $\alpha\in\(\frac 1{n+1},\frac 1n\)$ and
$$f(x):=x_1^k\1_{\R_+}(x_1).$$
Then we obtain for the Taylorlet transform of $f$ that
\begin{align*}
\mathcal{T}_\tau^{(n,\alpha)}f(a,0,0) &= a^{-(1+\alpha)/2}\cdot\int_{\R^2}f(x)\cdot\tau\cvec{x_1/a \\ x_2/a^\alpha}dx \\
&= a^{(1+\alpha)/2}\cdot \int_\R (ay_1)^k \1_{\R_+}(y_1)\cdot g(y_1) dy_1 \cdot \int_\R h(y_2)dy_2 \\
&= a^{k+(1+\alpha)/2} \cdot \underbrace{\int_0^\infty y_1^k\cdot g(y_1) dy_1}_{=0} \cdot \int_\R h(y_2)dy_2 = 0.
\end{align*}

\end{remark}

\begin{remark}

Furthermore, we want to emphasize the significance of the choice of $\alpha$ for the Taylorlet transform.

As the general setup involves the least common multiple $v_n$ of the numbers $1,\ldots, n$, it is possible that the order of the Taylorlet is higher than originally intended. For instance, if $\tau$ is an analyzing Taylorlet of order $5$ built after the general setup, we have $v_5=v_6=60$. To this end, Theorem \ref{construction} states that $\tau$ is also an analyzing Taylorlet of order $6$. 

The problems that arise from a wrong choice of $\alpha$ become clear when we consider a case where $\alpha<\frac 16$, $f$ is a $j-$feasible function and $\tau$ is the analyzing Taylorlet of order 5 (and 6) described above. If the highest approximation order of $s\in\R^6$ is $5$, we can treat the Taylorlet transform $\mathcal{T}^{5,\alpha}f(a,s,t)$ like $\mathcal{T}^{6,\alpha}f(a,\sigma,t)$, where $\sigma=(s_0,\ldots,s_5,0)$. We are allowed to do so, because for all $k\in\N$ we can write every shearing operator $S_s^{(k)}$ of order k as shearing operator $S_{s'}^{(k+1)}$ where $s'=(s_0,\ldots,s_k,0)$. Let $t\in\R$. If the highest approximation order of $\sigma$ for $f$ in $t$ is $5$, all conditions of Theorem \ref{main}, 2. are met and so the Taylorlet transform has a decay of $\O\(a^{(r-j)(1-6\alpha)+(\alpha-1)/2}\)$ for $a\to 0$. This can be significantly faster than the decay of $\sim a^{(\alpha-1)/2}$ for $a\to 0$ which occurs for the choice of $\alpha>\frac 16.$ \\

\end{remark}

\section{Proof of the main result}

In order to prove Theorem \ref{main}, we need the following auxiliary results.

\begin{lemma} \label{offset}
Let $f\in C(\R)$ such that for all $n\in\N$ there exists a constant $c_n\in\R_+$ with
$$\sup_{t\in\R} \left| t^n\cdot f(t) \right| =c_n<\infty.$$ 
Then
$$\int_{\R\setminus[-a^\beta,a^\beta]} f(t/a)dt = \O(a^N)\quad\text{for }a\to 0$$
for all $\beta<1$ and $N\in\N$.
\end{lemma}
\begin{proof}
By applying the decay condition, we obtain for $n>1$
$$\left|\int_{\R\setminus[-a^\beta,a^\beta]} f(t/a)dt\right| \le 2c_n \int_{a^\beta}^\infty (a/t)^n dt = \frac{2c_n}{n-1}a^{(1-\beta)n+\beta}.$$
Since we can choose $n\in\N$ arbitrarily large and since $\beta<1$, we get the desired result.
\flushright{$\Box$} \end{proof}

The next lemma provides a relation between the vanishing moments of order $n$ and the decay rate of integrals over the graph of a monomial. This will become important as the Taylorlet transform of a feasible function can be represented as a sum over integrals of this type.

\begin{lemma}\label{Radon}
Let $r,n\in\N$ and let $\tau$ be an analyzing Taylorlet of order $n$ with $r$ vanishing moments. Then for all $\ell,m\in\N$ and for all $k\in\{1,\ldots,n\}$, we have
\begin{equation}\label{int}
\int_\R \partial_1^m \tau\cvec{z\cdot t^k \\ t}t^{km+\ell}dt = \O\(|z|^{-(m+r)}\) \text{ for }z\to\pm\infty.
\end{equation}
\end{lemma}

\begin{proof}
The idea is to represent the integral in (\ref{int}) as a Fourier transform, to utilize the separation approach $\tau=g\otimes h$ and to show the decay result via the Fourier transforms of $g$ and $h$. We define the function
$$\tilde\tau_k(z,\omega):= \int_\R \tau\cvec{z\cdot t^k \\ t} e^{-it\omega} dt.$$
Then we can rewrite the left side of (\ref{int}) into
\begin{equation}\label{fourier}
\int_\R \partial_1^m \tau\cvec{z\cdot t^k \\ t}t^{km+\ell}dt = i^\ell\partial_\omega^\ell\partial_z^m\tilde\tau_k(z,0).
\end{equation}
For $k\in\N$ we introduce the function
$$g_{\pm,k}:\R\to\R,\quad  t\mapsto g(\pm t^k).$$
Due to the vanishing moment property we can conclude that
$$\widehat{g_{\pm,k}}^{(\nu)}(0)=0 \text{ for all }\nu\in\{0,\ldots,kr-1\}.$$
Consequently, we get the decay rate 
\begin{equation}\label{decay}
\widehat{g_{\pm,k}}^{(\nu)}(\omega)=\O(\omega^{kr-1-\nu})\quad\text{for }\omega\to 0.
\end{equation}
We now obtain
$$\tilde\tau_k(z,\omega) = \(\frac 1{|z|^{1/k}} \(g_{\sgn(z),k}\)^\land\(\frac \cdot{|z|^{1/k}}\)*\hat h\)(\omega)$$
and hence
$$\partial_\omega^\ell\tilde\tau_k(z,0) = \frac 1{|z|^{1/k}} \int_\R \(g_{\sgn(z),k}\)^\land\(-\frac \omega{|z|^{1/k}}\) \hat h^{(\ell)}(\omega)d\omega.$$
We will now check the decay rate of $\partial_\omega^\ell\partial_z^m\tilde\tau_k(z,0)$. For this, we observe that
\begin{align*}
\partial_\omega^\ell\partial_z^m\tilde\tau_k(z,0) &= \partial_z^m\(|z|^{-1/k}\int_\R \(g_{\pm,k}\)^\land\(-\frac \omega{|z|^{1/k}}\) \hat h^{(\ell)}(\omega)d\omega\) \\
&= \sum_{\nu=0}^m c_\nu |z|^{-[m+(\nu+1)/k]}\int_\R \omega^\nu \[\(g_{\sgn(z),k}\)^\land\]^{(\nu)}\(-\frac \omega{|z|^{1/k}}\)\hat h^{(\ell)}(\omega)d\omega
\end{align*}
with $c_\nu\in\R$ for all $\nu\in\{0,\ldots,m\}$. By applying (\ref{decay}) and $h\in\S(\R)$ we estimate the terms in this equation as
\begin{align*}
& \ \left| |z|^{-[m+(\nu+1)/k]}\int_\R \omega^\nu \[\(g_{\sgn(z),k}\)^\land\]^{(\nu)}\(-\frac \omega{|z|^{1/k}}\)\hat h^{(\ell)}(\omega)d\omega \right| \\
\le &\ 2|z|^{-[m+(\nu+1)/k]}\int_0^\infty \omega^\nu\cdot \(\omega |z|^{-1/k}\)^{kr-1-\nu}\cdot\frac 1{1+\omega^{kr+1}}d\omega \\
= &\ c\cdot |z|^{-(m+r)}
\end{align*}
for some constant $c>0$.
\\ \flushright{$\Box$} \end{proof}

The next lemma's statement is essentially the same as in Theorem \ref{main}, but we restrict the choice of analyzed functions to 0-feasible functions.

\begin{lemma}\label{aux}
Let $r,n\in\N$ and let $\tau$ be an analyzing Taylorlet of order $n$ with $r$ vanishing moments. Let furthermore $t\in\R$ and let $f$ be a 0-feasible function.
\begin{enumerate}

\item Let $\alpha>0$. If $s_0\ne q(t)$, the Taylorlet transform has a decay of
$$\mathcal{T}^{(n,\alpha)}f(a,s,t)=\O\big(a^N\big)\quad\text{for }a\to 0$$
for all $N>0$.

\item Let $\alpha<\frac 1n$ and let $k\in\{0,\ldots,n-1\}$ be the highest approximation order of $s$ for $f$ in $t$. Then the Taylorlet transform has the decay property
$$\mathcal{T}^{(n,\alpha)}f(a,s,t)=\O\(a^{(\alpha-1)/2+r[1-(k+1)\alpha]}\)\quad\text{for }a\to 0.$$

\item Let $\alpha>\frac 1{n+1}$, let $f$ be a $(j,n+1)-$feasible function and let $\tau$ be restrictive. If $n$ is the highest approximation order of $s$ for $f$ in $t$, then the Taylorlet transform has the decay property
$$\mathcal{T}^{(n,\alpha)}f(a,s,t)\sim a^{(\alpha-1)/2}\quad\text{for }a\to 0.$$
\end{enumerate}
\end{lemma}

\begin{proof} We restrict ourselves to the case $t=0$ as every other case results from a shift. Furthermore, we note that $f\in \S'(\R^2)$. SInce, the Taylorlet transform $\mathcal{T}^{(n,\alpha)}f(a,s,t)=a^{(1+\alpha)/2}\Langle \tau_{a,s,0}, f\Rangle$ is well defined. \\[3mm]
{\it 1.} The idea is to exploit the special form of $f$ in order to simplify its Taylorlet transform and to use the Schwartz class decay condition of $\tau$ in order to estimate the integral.

The structure of $f$ leads to the following form of the Taylorlet transform:
\begin{align} \label{Taylorlet}
\mathcal{T}^{(n,\alpha)} f(a,s,0) &= \int_{\R^2}\delta(x_1-q(x_2))\tau_{a,s,0}(x)dx \nonumber \\
&= a^{-(1+\alpha)/2} \int_\R \tau \cvec{\[q(x_2)-\sum_{\ell=0}^n \frac{s_\ell}{\ell !}\cdot  x_2^\ell\]/a \\ x_2/a^\alpha}dx_2 \\
&= a^{-(1+\alpha)/2} \int_\R g \(\tilde q(x_2)/a \) h( x_2/a^\alpha)dx_2 \nonumber,
\end{align}
where $\tilde q(x_2)=q(x_2)-\sum_{k=0}^n \frac{s_\ell}{\ell !}\cdot  x_2^k$. Since $g,h\in\S(\R)$, the integrand in the last row fulfills the necessary decay condition of Lemma \ref{offset}. By applying this lemma, we can conclude that
$$\left|\int_{\R\setminus[-a^\beta,a^\beta]}g \(\tilde q(x_2)/a \) h( x_2/a^\alpha)dx_2 \right| \le \|g\|_{L^\infty}\cdot \left|\int_{\R\setminus[-a^\beta,a^\beta]}h( x_2/a^\alpha)dx_2 \right| \stackrel{\text{Lemma \ref{offset}}}{=}\O(a^N)\quad\text{for }a\to 0$$
for all $N\in\N$ if $\beta<\alpha$. Hence,
$$\mathcal{T}^{(n,\alpha)} f(a,s,0) = a^{-(1+\alpha)/2} \int_{-a^\beta}^{a^\beta} g \(\tilde q(x_2)/a \) h( x_2/a^\alpha)dx_2 + \O(a^N)\quad\text{for }a\to 0$$
for all $N>0$. Due to the conditions of this lemma, $\tilde q\in C(\R)$ and $\tilde q(0)\ne 0$. Hence, there exists an $\e>0$ such that $d:=\min_{x_2\in[-\e,\e]}|\tilde q(x_2)| > 0.$ By employing the boundedness of $h$ and the Schwartz decay condition $\sup_{x_2\in\R}|x_2^M\cdot g(x_2)|=c_M<\infty$, we get
\begin{align*}
|\mathcal{T}^{(n,\alpha)} f(a,s,0)| &\le c_M a^{-(1+\alpha)/2} \int_{-a^\beta}^{a^\beta} \(\frac a{|\tilde q(x_2)|}\)^M dx_2 \\
&\le 2c_M d^{-M}  a^{M+\beta-(1+\alpha)/2}.
\end{align*}
Since we can choose $M$ to be arbitrarily large, the result follows immediately. \\[3mm]
{\it 2.} The general idea of this proof is to represent the Taylorlet transform as a sum of integrals of the form (\ref{int}) and to apply Lemma \ref{Radon} in order to obtain the desired decay rate. To this end, we will divide the proof into four steps. \\[3mm]
{\sc Step 1}

In the first step we will show that the Taylorlet transform $\mathcal{T}^{(n,\alpha)}f(a,s,0)$ is an integral over a curve and we will prove that only a small neighborhood of the origin is relevant for the decay of the Taylorlet transform for $a\to 0$.

First, we rewrite (\ref{Taylorlet}):
\begin{align}\label{basic}
a^{(1+\alpha)/2}\mathcal{T}^{(n,\alpha)} f(a,s,0) &= \int_\R \tau\cvec{\[q(x_2)-\sum_{\ell=0}^n \frac{s_\ell}{\ell !}\cdot  x_2^\ell\]/a \\ x_2/a^\alpha} dx_2 \nonumber \\
&= \int_\R \tau\cvec{\tilde q(x_2) \cdot x_2^{k+1}/a \\ x_2/a^\alpha} dx_2,
\end{align}
where $\tilde q(x_2)=\begin{cases} 
x_2^{-(k+1)}\[q(x_2)-\sum_{\ell=0}^n \frac{s_\ell}{\ell !}\cdot  x_2^\ell\]  & \text{for }x_2 \ne 0, \\
\frac 1{(k+1)!}\cdot\(q^{(k+1)}(0) - s_{k+1}\) & \text{for }x_2 = 0.
\end{cases}$ \\[2mm]
Since $k$ is the highest approximation order of $s$ for $f$ in $t=0$, we have $s_{k+1}\ne q^{(k+1)}(0)$. Hence, $\tilde q(0)\ne 0$. Furthermore, we have $\tilde q\in C^\infty(\R)$ due to the prerequisites. In order to show that just a small neighborhood of the origin is responsible for the decay of the Taylorlet transform for $a\to 0$, we observe that the integrand in (\ref{basic}) fulfills the decay condition of Lemma \ref{offset}. By applying this lemma, we obtain for $\beta\in\(0,\frac 1{k+1}\)$ and an arbitrary $N\in\N$ that
\begin{equation}\label{trunc}\big|a^{(1+\alpha)/2}\mathcal{T}^{(n,\alpha)} f(a,s,0)\big| = \left|\int_{-a^\beta}^{a^\beta} \tau\cvec{\tilde q(x_2) \cdot x_2^{k+1}/a \\ x_2/a^\alpha} dx_2\right| +\O(a^N)\quad\text{for }a\to 0.\end{equation}
\newpage
\noindent{\sc Step 2}

If we replaced the term $\tilde q(x_2)$ in the argument of the integrand by some constant $c\ne 0$, the integral would be a truncated version of the desired form (\ref{int}). Hence, we could apply Lemma \ref{Radon} to obtain an estimate for the decay of the Taylorlet transform. In order to get closer to this form, we will approximate the integrand by a Taylor polynomial in this step.

Now we expand the integrand of (\ref{trunc}) into a Taylor series \dtt wrt to the first component in a neighborhood of the point $\tilde q(0)\cdot x_2^{k+1}/a$.
\begin{align}\label{Rest}
\big|a^{(1+\alpha)/2}\mathcal{T}^{(n,\alpha)} f(a,s,0)\big| &= \left|\int_{-a^\beta}^{a^\beta} \tau\cvec{\tilde q(x_2) \cdot x_2^{k+1}/a \\ x_2/a^\alpha} dx_2\right| +\O(a^N) \nonumber \\
& \le \sum_{m=0}^M \left|\int_{-a^\beta}^{a^\beta} \partial_1^m\tau\cvec{\tilde q(0)\cdot x_2^{k+1}/a \\ x_2/a^\alpha}\cdot \(\frac{x_2^{k+1}}a\)^m\cdot \frac{\[\tilde q(x_2)-\tilde q(0)\]^m}{m!} dx_2\right| \nonumber \\
&\qquad + c^{M+1}\int_{-a^\beta}^{a^\beta}\(\frac{|x_2|^{k+1}}a\)^{M+1}\cdot |x_2|^{M+1}dx_2 +\O(a^N).
\end{align}
For the last estimate we used that $|\tilde q(x_2)-\tilde q(0)|\le c|x_2|$ for all $x_2\in\R$. We now prove that it is possible to choose $M\in\N$ such that the rest term (\ref{Rest}) behaves like $\O(a^N)$ for $a\to 0$ for an arbitrary, but fixed $N\in\N$. We obtain
\begin{equation}\label{rate}
\int_{-a^\beta}^{a^\beta}\(\frac{|x_2|^{k+1}}a\)^{M+1}\cdot |x_2|^{M+1}dx_2 \sim a^{(M+1)(\beta(k+2)-1)+\beta} \quad\text{for }a\to 0.
\end{equation}
By restricting the choice of $\beta\in\(0,\frac 1{k+1}\)$ to $\beta\in\(\frac 1{k+2},\frac 1{k+1}\)$, we obtain the desired decay rate of $\O(a^N)$ for
$$M=\left\lceil\frac{N-\beta}{\beta(k+2)-1}\right\rceil-1.$$
{\sc Step 3}

In the third step we will expand $\tilde q$ in a Taylor series about the origin to obtain a representation of the Taylorlet transform as a sum of truncated versions of integrals of the form (\ref{int}).

We now expand the term $\tilde q(x_2)-\tilde q(0)$ about the point $x_2=0$.
\begin{align}\label{ugly}
&\big|a^{(1+\alpha)/2}\mathcal{T}^{(n,\alpha)} f(a,s,0)\big| \nonumber \\
\le &\sum_{m=0}^M \left|\int_{-a^\beta}^{a^\beta} \partial_1^m\tau\cvec{\tilde q(0)\cdot x_2^{k+1}/a \\ x_2/a^\alpha}\cdot \(\frac{x_2^{k+1}}a\)^m\cdot \frac{\[\tilde q(x_2)-\tilde q(0)\]^m}{m!} dx_2\right| +\O(a^N) \nonumber \\
\le &\sum_{m=0}^M\frac 1{m!}\left| \int_{-a^\beta}^{a^\beta} \partial_1^m\tau\cvec{\tilde q(0) \cdot x_2^{k+1}/a \\ x_2/a^\alpha} \(\frac{x_2^{k+1}}a\)^m\cdot \[ \rho(x_2) +\sum_{\ell=1}^{L_m}\frac {\tilde q^{(\ell)}(0)}{\ell!}\cdot x_2^\ell\]^m dx_2\right| +\O(a^N) \quad\text{for }a\to 0,
\end{align}
where $\rho(x_2)$ is the rest term of the Taylor series expansion with the property $\rho(x_2)=\O\big(x_2^{L_m+1}\big)$ for $x_2\to0$. Now we estimate the summands for each $m\in\{0,\ldots,M\}$.

\begin{align*}
& \left| \int_{-a^\beta}^{a^\beta} \partial_1^m\tau\cvec{\tilde q(0) \cdot x_2^{k+1}/a \\ x_2/a^\alpha} \(\frac{x_2^{k+1}}a\)^m\cdot \[\rho(x_2) + \sum_{\ell=1}^{L_m} \frac {\tilde q^{(\ell)}(0)}{\ell!}\cdot x_2^\ell\]^m dx_2\right| \\
\le& \sum_{\nu=0}^m\Bigg| \int_{-a^\beta}^{a^\beta} \partial_1^m\tau\cvec{\tilde q(0) \cdot x_2^{k+1}/a \\ x_2/a^\alpha} \(\frac{x_2^{k+1}}a\)^m\cdot{{m}\choose{\nu}} \cdot [\rho(x_2)]^\nu\cdot\[\sum_{\ell=1}^{L_m} \frac {\tilde q^{(\ell)}(0)}{\ell!}\cdot x_2^\ell\]^{m-\nu} dx_2\Bigg|.
\end{align*}
Since $\tau\in\S(\R^2)$, $\rho(x_2)=\O(x_2^{L_m+1})$ for $x_2\to 0$ and $\sum_{\ell=1}^{L_m} \frac {\tilde q^{(\ell)}(0)}{\ell!}\cdot x_2^\ell =\O(x_2)$ for $x_2\to 0$, for every $\nu\in\{1,\ldots,m\}$ there exist $c_\nu,a_0>0$ such that for all $a>a_0$
\begin{align*}
&\ \Bigg| \int_{-a^\beta}^{a^\beta} \partial_1^m\tau\cvec{\tilde q(0) \cdot x_2^{k+1}/a \\ x_2/a^\alpha} \(\frac{x_2^{k+1}}a\)^m\cdot{{m}\choose{\nu}} \cdot [\rho(x_2)]^\nu\cdot\[\sum_{\ell=1}^{L_m} \frac {\tilde q^{(\ell)}(0)}{\ell!}\cdot x_2^\ell\]^{m-\nu} dx_2\Bigg| \\
\le &\ c_\nu\cdot a^{-m} \Bigg| \int_{-a^\beta}^{a^\beta} |x_2|^{(k+1)m}|x_2|^{(L_m+1)\nu}\cdot|x_2|^{m-\nu}  dx_2\Bigg| \\
= &\ \O\big( a^{[(k+2)\beta-1]m+\beta(L_m+2)} \big) \quad\text{for }a\to 0.
\end{align*}
By comparing this decay rate to (\ref{rate}) which is equal to $\O(a^N)$ and considering that $\beta>(k+2)\beta -1$, we see that the choice $L_m=M-m$ is sufficient for the decay rate of $\O(a^N)$. Hence, for all $m\in\{0,\ldots,M\}$ we have
\begin{align*}
& \ \left| \int_{-a^\beta}^{a^\beta} \partial_1^m\tau\cvec{\tilde q(0) \cdot x_2^{k+1}/a \\ x_2/a^\alpha} \(\frac{x_2^{k+1}}a\)^m\cdot \[\rho(x_2) + \sum_{\ell=1}^{L_m} \frac {\tilde q^{(\ell)}(0)}{\ell!}\cdot x_2^\ell\]^m dx_2\right| \\
\le& \ \Bigg| \int_{-a^\beta}^{a^\beta} \partial_1^m\tau\cvec{\tilde q(0) \cdot x_2^{k+1}/a \\ x_2/a^\alpha} \(\frac{x_2^{k+1}}a\)^m \cdot\[\sum_{\ell=1}^{L_m} \frac {\tilde q^{(\ell)}(0)}{\ell!}\cdot x_2^\ell\]^{m} dx_2\Bigg| +\O(a^N) \quad\text{for }a\to 0.
\end{align*}
By inserting this result into (\ref{ugly}), we get
\begin{align*}
 &\ a^{(1+\alpha)/2}\mathcal{T}^{(n,\alpha)} f(a,s,0) \\
 = &\ \sum_{m=0}^M a^{-m}\sum_{\ell=m}^{(M-m)m} c_{\ell,m}\int_{-a^\beta}^{a^\beta}\partial_1^m\tau\cvec{\tilde q(0) \cdot x_2^{k+1}/a \\ x_2/a^\alpha}\cdot x_2^{(k+1)m+\ell} dx_2 + \O(a^N) \quad\text{for }a\to 0
\end{align*}
for appropriate constants $c_{\ell,m}\in\R$.\\[3mm]
{\sc Step 4}

In this final step, we extend the integration limits to $\pm\infty$ and apply Lemma \ref{Radon} to estimate the decay of the Taylorlet transform.

Applying Lemma \ref{offset} again, we can change back the integration limits to $\pm\infty$ by only adding another $\O(a^N)$-term. Furthermore, we substitute $x_2=a^\alpha v$ and obtain
\begin{align}\label{summary}
& \mathcal{T}^{(n,\alpha)} f(a,s,0) \nonumber \\
= & a^{(\alpha-1)/2} \sum_{m=0}^M a^{-m}\sum_{\ell=m}^{(M-m)m} c_{\ell,m}\int_\R\partial_1^m\tau\cvec{\tilde q(0) \cdot a^{(k+1)\alpha-1}v^{k+1} \\ v}\cdot (a^\alpha v)^{(k+1)m+\ell} dv + \O(a^N).
\end{align}
Finally, we brought the Taylorlet transform into a shape that is fit for an application of Lemma \ref{Radon}. Since $\lim_{a\to 0}a^{(k+1)\alpha-1}=\infty$, we obtain
\begin{align*}
|\mathcal{T}^{(n,\alpha)} f(a,s,0)| &= \O\( a^{(\alpha-1)/2}\sum_{m=0}^M  a^{-m}\sum_{\ell=m}^{(M-m)m} c_{\ell,m}\cdot a^{\alpha\ell}\cdot a^{\alpha(k+1)m}\cdot a^{(r+m)(1-(k+1)\alpha)}\) \\
&=\O\( \sum_{m=0}^M\sum_{\ell=m}^{(M-m)m} c_{\ell,m}\cdot a^{(\alpha-1)/2 +\alpha\ell+(1-(k+1)\alpha)r} \) \\
&=\O\(a^{(1-(k+1)\alpha)r+(\alpha-1)/2}\)\quad\text{for }a\to 0.
\end{align*}
{\it 3.} For this case we use the same argumentations as in the case {\it 2.} to obtain (\ref{summary}) with the choices of $k=n$ and
$$\tilde q(x_2) = \begin{cases}
x_2^{-(n+1)}\cdot\[q(x_2)-\sum_{\ell=0}^n \frac{s_\ell}{\ell !}\cdot  x_2^\ell\], & \text{for } x_2\ne 0, \\
\frac 1{(n+1)!}\cdot q^{(n+1)}(0), & \text{for } x_2 = 0.
\end{cases}$$
In spite of the similarities there is a major difference in the situations, namely that 
$$\lim_{a\to 0} a^{(n+1)\alpha-1}=0.$$
Hence, we obtain for the integrals in (\ref{summary}) that
\begin{align}\label{rest}
& \lim_{a\to 0} \int_\R \partial_1^m \tau \cvec{\tilde q(0) a^{(n+1)\alpha-1}u^{n+1} \\ u} u^{(n+1)m+\ell} du \nonumber \\
=& \int_\R \lim_{a\to 0} g^{(m)}\big( \tilde q(0) a^{(n+1)\alpha-1}u^{n+1} \big) h(u) u^{(n+1)m+\ell} du \nonumber \\
=& g^{(m)}(0) \int_\R h(u)u^{(n+1)m+\ell} du.
\end{align}
We now focus on the powers of $a$ appearing in the summands of (\ref{summary}). For the indices $\ell$ and $m$ of the double sum's summands in (\ref{summary}) we obtain that
\begin{align*}
S_{\ell,m}(a) & :=a^{(\alpha-1)/2-m} \int_\R\partial_1^m\tau\cvec{\tilde q(0) \cdot a^{(k+1)\alpha-1}v^{k+1} \\ v}\cdot (a^\alpha v)^{(k+1)m+\ell} dv \\
&= \O\( a^{(\alpha-1)/2+[(n+1)\alpha-1]m+\ell\alpha}\)\quad \text{for }a\to 0
\end{align*}
for all $m\in\{0,\ldots,M\}$ and $\ell\in\{m,\ldots,(M-m)m\}$. Due to the restrictiveness $g(0)\ne 0$ and $\int_\R h(u)du\ne 0$. Hence, we obtain with (\ref{rest}) that
$$S_{0,0}(a) \sim a^{(\alpha-1)/2}\cdot g(0)\cdot\int_\R h(u)du \sim a^{(\alpha-1)/2} \quad\text{for }a\to 0.$$
Since $(n+1)\alpha-1>0$, $S_{0,0}$ is the slowest decaying summand. Thus,
$$\mathcal{T}^{(n,\alpha)} f(a,s,0) \sim a^{(\alpha-1)/2}\quad\text{for }a\to 0.$$
\\ \flushright{$\Box$} \end{proof}

With this lemma we are now able to prove Theorem \ref{main}.

\begin{proof}[of Theorem \ref{main}] ~

The proof strategy is to reduce the case $f(x)=I_\pm^j \delta(x_1-q(x_2))$ to the case $f(x)= \delta(x_1-q(x_2))$ of Lemma \ref{aux} by partial integration and to show that the resulting iterated integral $I_\pm^j\tau$ of the Taylorlet $\tau$ is a Taylorlet as well.

First, we note that $f = (I^j_\pm \delta) (\cdot_1-q(\cdot_2))$ is a tempered distribution for all $j\in\N$. Let $a>0,s\in \R^{n+1},t\in\R$. Then the Taylorlet transform $\mathcal{T}^{(n,\alpha)}f(a,s,t)$ is well defined. By partial integration we obtain for $j\ge 1$
\begin{align*}
a^{(1+\alpha)/2}\cdot\mathcal{T}^{(n,\alpha)}f(a,s,t) &= \Langle \tau_{a,s,0}, f\Rangle \\
&= \int_\R \[a\cdot I_{x_1,\pm}\tau_{a,s,0}(x)\cdot (I_\pm^j\delta)(x_1-q(x_2))\]_{x_1=-\infty}^{x_1=+\infty}dx_2 \\
&\quad +\int_{\R^2}a\cdot I_{x_1,\mp}\tau_{a,s,0}(x)\cdot (I_\pm^{j-1}\delta)(x_1-q(x_2))dx.
\end{align*}
We now show that the first term disappears. For this we note that $I_\pm^j \delta(x)$ exhibits only polynomial growth as $|x|\to\infty$. Furthermore, with $\tau= g\otimes h$, only $g$ is altered by the operator $I_{x_1,\pm}$ while $h$ remains the same. Hence, we show that $I_\pm^j g\in\S(\R)$ for all $j<r$. By applying a Fourier transform to the function, we obtain
$$\(I_\pm^j g\)^\land(\omega) = \frac{\hat g(\omega)}{(\pm i \omega)^j}.$$
Since $g$ has $r$ vanishing moments and $g\in\S(\R)$, we have $\hat g(\omega)=\O(\omega^r)$ for $\omega\to 0$. Consequently, $\frac{\hat g(\omega)}{(\pm i \omega)^j}\in\S(\R)$ and hence also $I_\pm^j g\in\S(\R).$ We thus obtain
$$a^{(1+\alpha)/2}\cdot\mathcal{T}^{(n,\alpha)}f(a,s,t) = a\cdot \int_{\R^2} I_{x_1,\mp}\tau_{a,s,0}(x)\cdot (I_\pm^{j-1}\delta)(x_1-q(x_2))dx.$$
By induction we get
$$a^{(1+\alpha)/2}\cdot\mathcal{T}^{(n,\alpha)}f(a,s,t) = a^j\cdot \Langle I_{x_1,\mp}^j\tau_{a,s,0}\ ,\ \delta(x_1-q(x_2))\Rangle.$$
This delivers an additional factor $a^j$ to the Taylorlet transform. Now we examine the vanishing moments. By applying partial integration and utilizing $I_\pm^j g\in\S(\R)$ we obtain
$$\left|\int_\R (I_\pm^j g)(\pm t^k)t^m dt\right| = \left|\int_\R g(\pm t^k) t^{kj+m}dt\right|.$$
Hence, $I_\pm^j g$ has $r-j$ vanishing moments of order $n$ and in case $2.$ with a highest approximation order of $k$ we obtain the decay rate
$$\mathcal{T}^{(n,\alpha)}f(a,s,0)=\O\(a^{j+(\alpha-1)/2+(r-j)[1-(k+1)\alpha]}\)\text{ for }a\to 0.$$

It remains to show that the restrictiveness condition of $g$ guarantees $I_\pm^j g(0)\ne 0$. For this we apply the formula for iterated integrals stating that
$$I_+^j g(u) = \int_{-\infty}^u (u-v)^{j-1} g(v) dv.$$
Hence, we obtain
$$I_+^j g(0) = (-1)^{j-1} \int_{-\infty}^0 g(v)v^{j-1} dv = \underbrace{\int_\R g(v)v^{j-1} dv}_{=0} + \ (-1)^j\underbrace{\int_0^\infty g(v)v^{j-1} dv}_{\ne 0}\ne 0.$$
The statement $I_-^j g(0)\ne 0$ can be proved similarly. Hence, for all $j<r$, the function $I_{x_1,\pm}^j\tau$ is a restrictive analyzing Taylorlet of order $n$ with $r-j$ vanishing moments, \dt ie,
$$\int_\R I_{x_1,\pm}^j\tau\cvec{0\\u}du\ne 0.$$
Consequently, with the additional factor $a^j$ we get the decay rate
$$\mathcal{T}^{(n,\alpha)}f(a,s,0)=\O\(a^{j+(\alpha-1)/2}\)\text{ for }a\to 0.$$
\\ \flushright{$\Box$} \end{proof}

\section{Numerical results}

In this section we verify the main result numerically. To this end, we implemented the Taylorlet transform in Matlab and used Taylorlets of order 2 with 3 vanishing moments. They were constructed via the General Setup in section 3 starting from the function $\phi(t)=e^{-t^2}$. Through this procedure we obtain the Taylorlet
\begin{align}\label{example}
\tau(x)&=g(x_1)\cdot h(x_2),  \intertext{where} 
g(x_1)&= \frac{64}{8!}\cdot(1+x_1)\cdot(315 - 51660x_1^2 + 286020x_1^4 - 349440x_1^6 + 142464x_1^8 - 21504x_1^{10} + 1024x_1^{12})\cdot e^{-x_1^2},\nonumber \\ 
h(x_2)&=e^{-x_2^2}, \nonumber
\end{align}
which is shown in Figure \ref{Taylorfig}. To speed up the computation time, we employed the one-dimensional adaptive Gauss-Kronrod quadrature \texttt{quadgk} for the evaluation of the integrals. In order to transfer the integrals into a scenario in which a one-dimensional integration is possible, the reduction technique from the proof of part 3. of \Cref{main} was utilized to reduce the integrals to the cases of $0-$feasible functions, \dt ie,
$$\Langle \tau_{ast}(x), \1_{\R_+}(x_1-q(x_2))\Rangle = \Langle I_{x_1,+}\tau_{ast},\delta(x_1-q(x_2)) \Rangle.$$
In this example, the antiderivative of $\tau$ \dtt wrt $x_1$ can be determined analytically by computing the antiderivative of $g$, \dt ie,
\begin{align*}
\int_{-\infty}^t g(x_1)dx_1 &= -\frac{32}{8!}\cdot e^{-t^2}\cdot \big(-9 - 630 t - 324 t^2 + 34020 t^3 + 25668 t^4 - 100800 t^5 - 86784 t^6 + 71040 t^7 \\
&\hspace{25mm} + 65664 t^8 - 15872 t^9 - 15360 t^{10} + 1024 t^{11} + 1024 t^{12}\big).
\end{align*}

\begin{figure}[p]
\begin{minipage}{.49\textwidth}
\centering
\includegraphics[width=.95\textwidth]{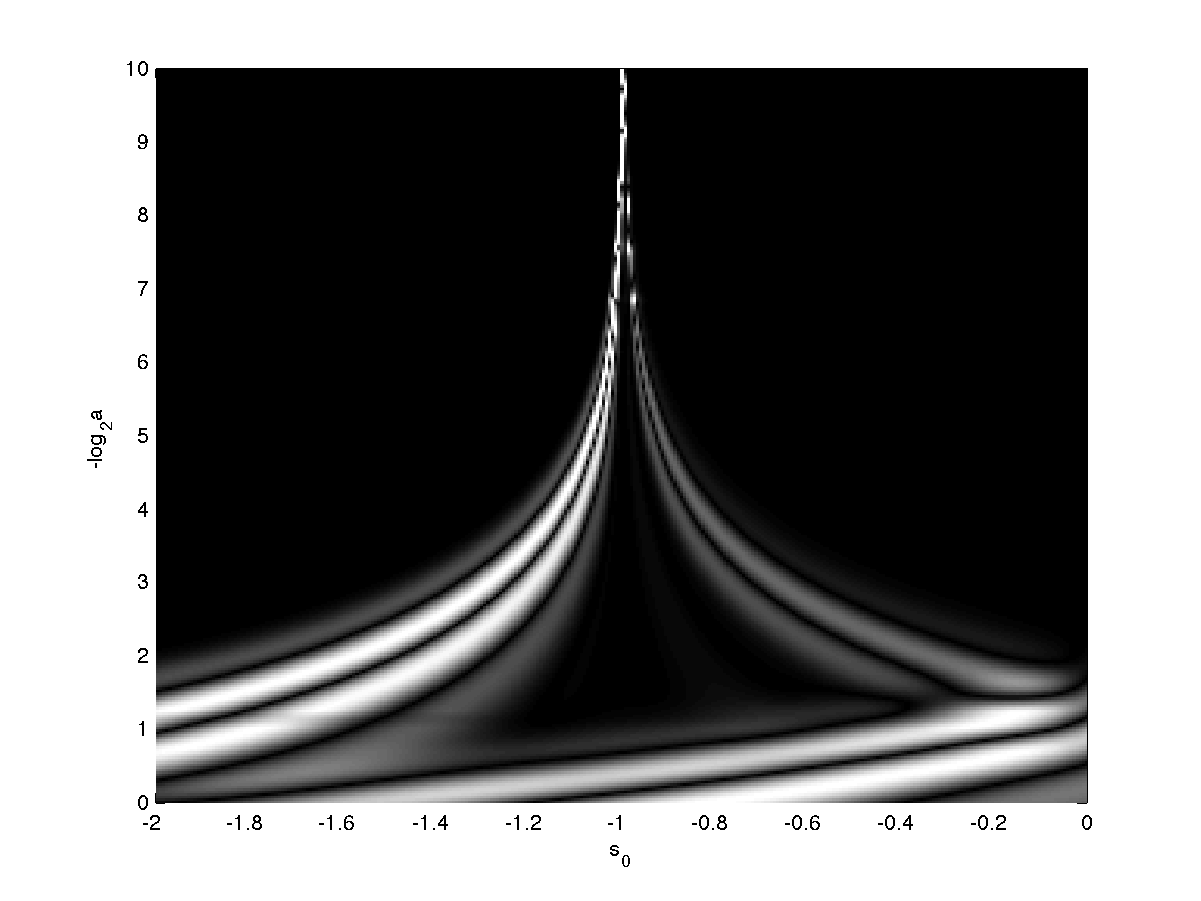}
\includegraphics[width=.95\textwidth]{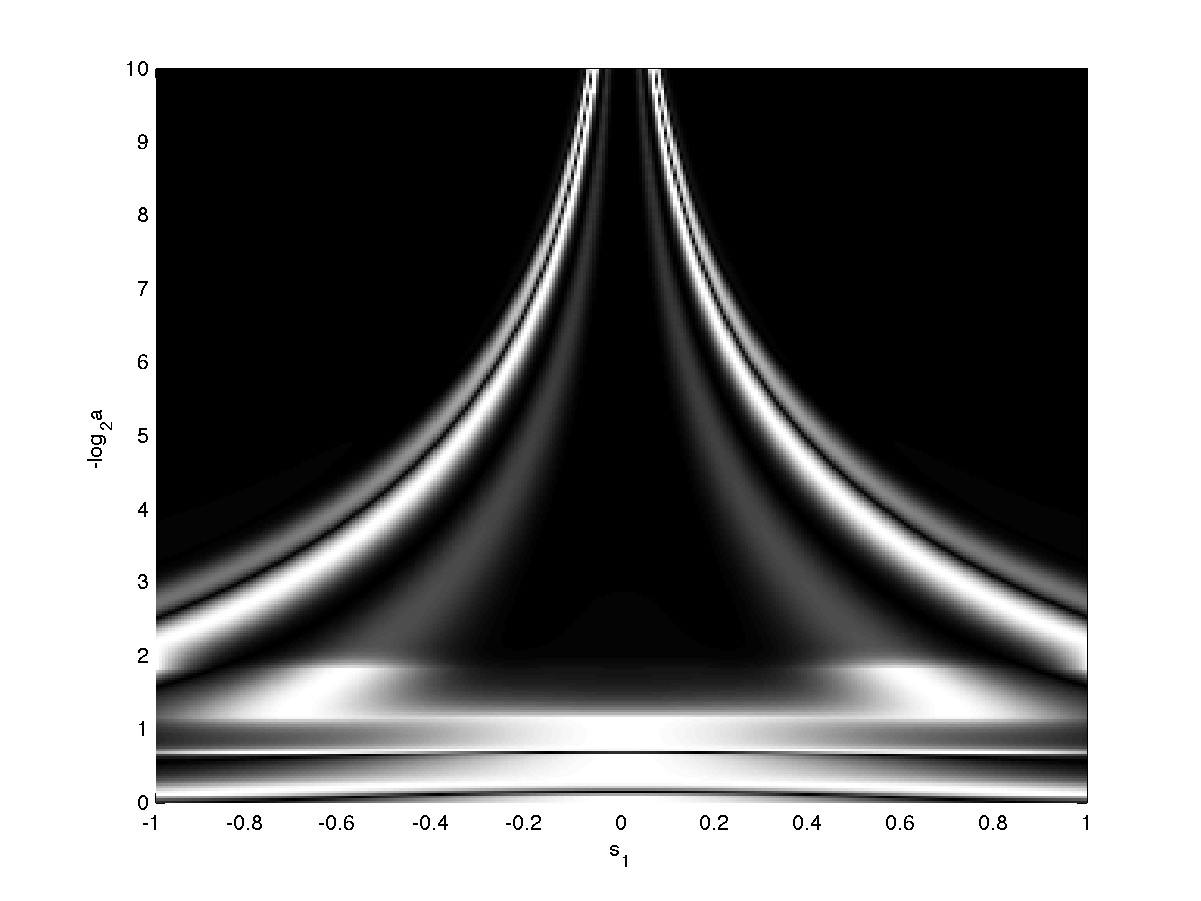}
\includegraphics[width=.95\textwidth]{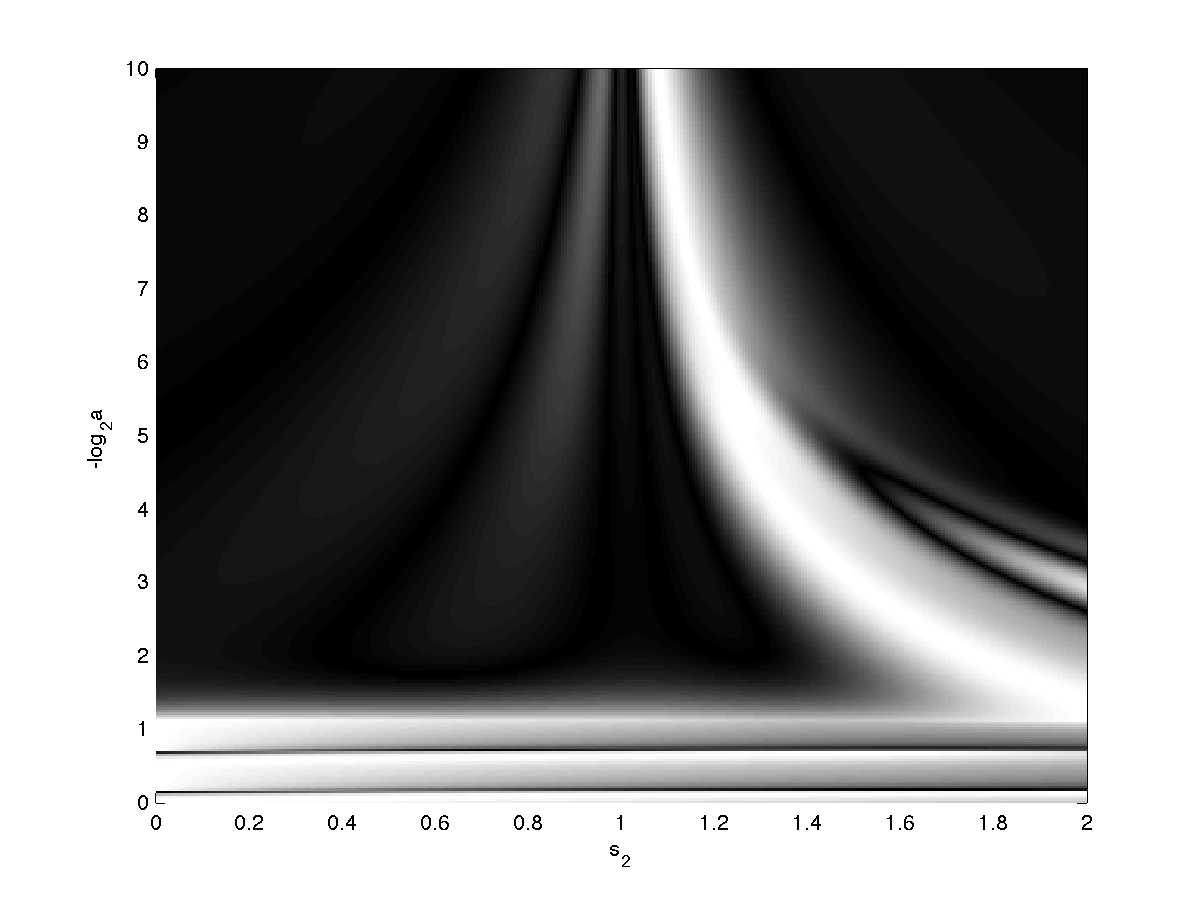}
\end{minipage}
\hfill
\begin{minipage}{.49\textwidth}
\centering
\includegraphics[width=.95\textwidth]{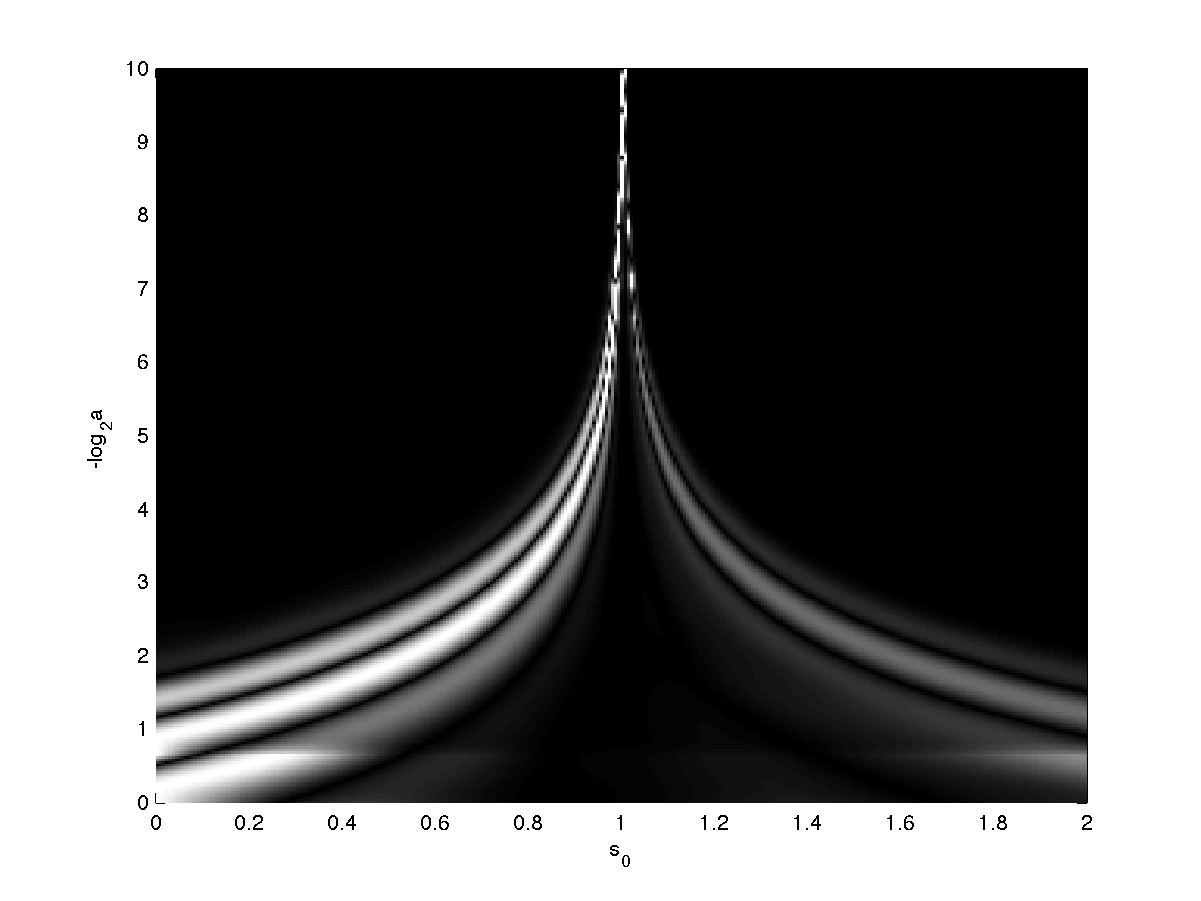}
\includegraphics[width=.95\textwidth]{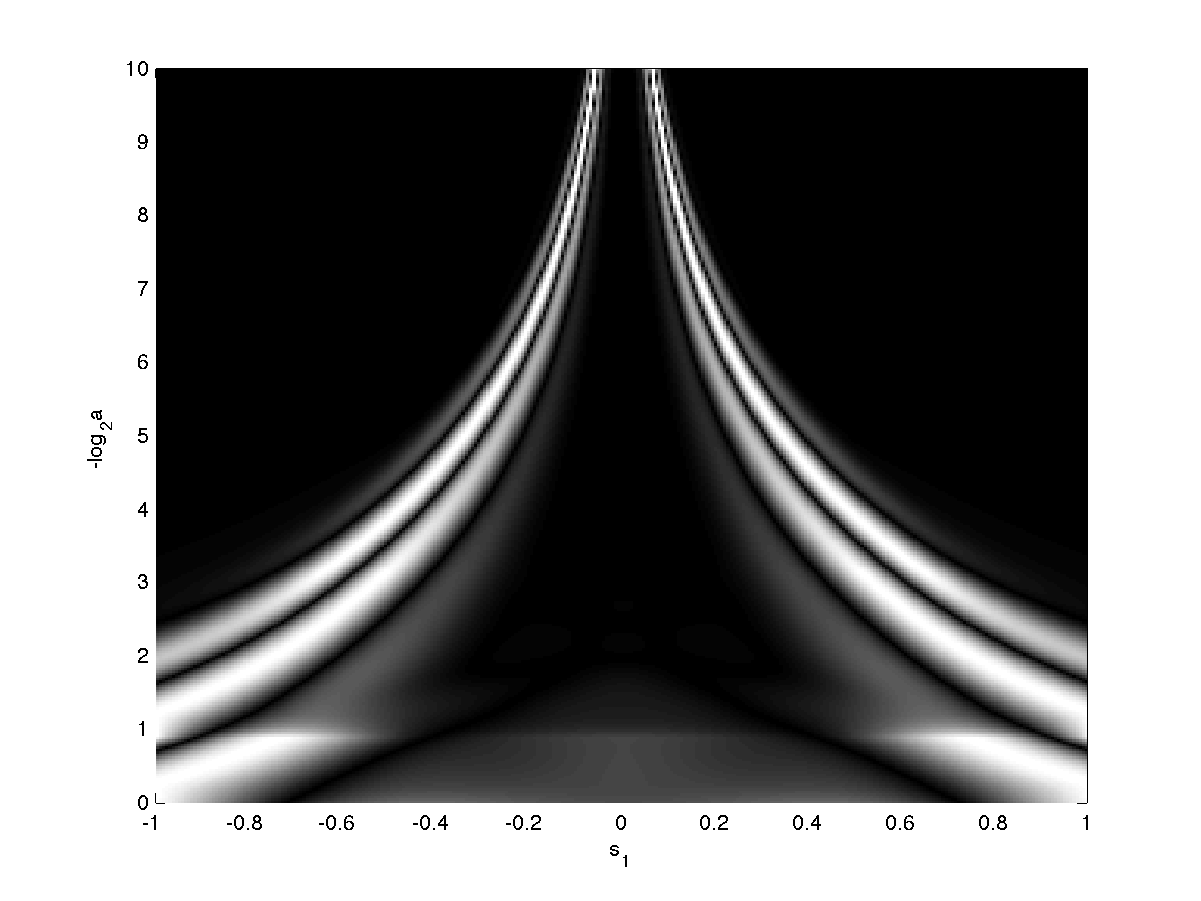}
\includegraphics[width=.95\textwidth]{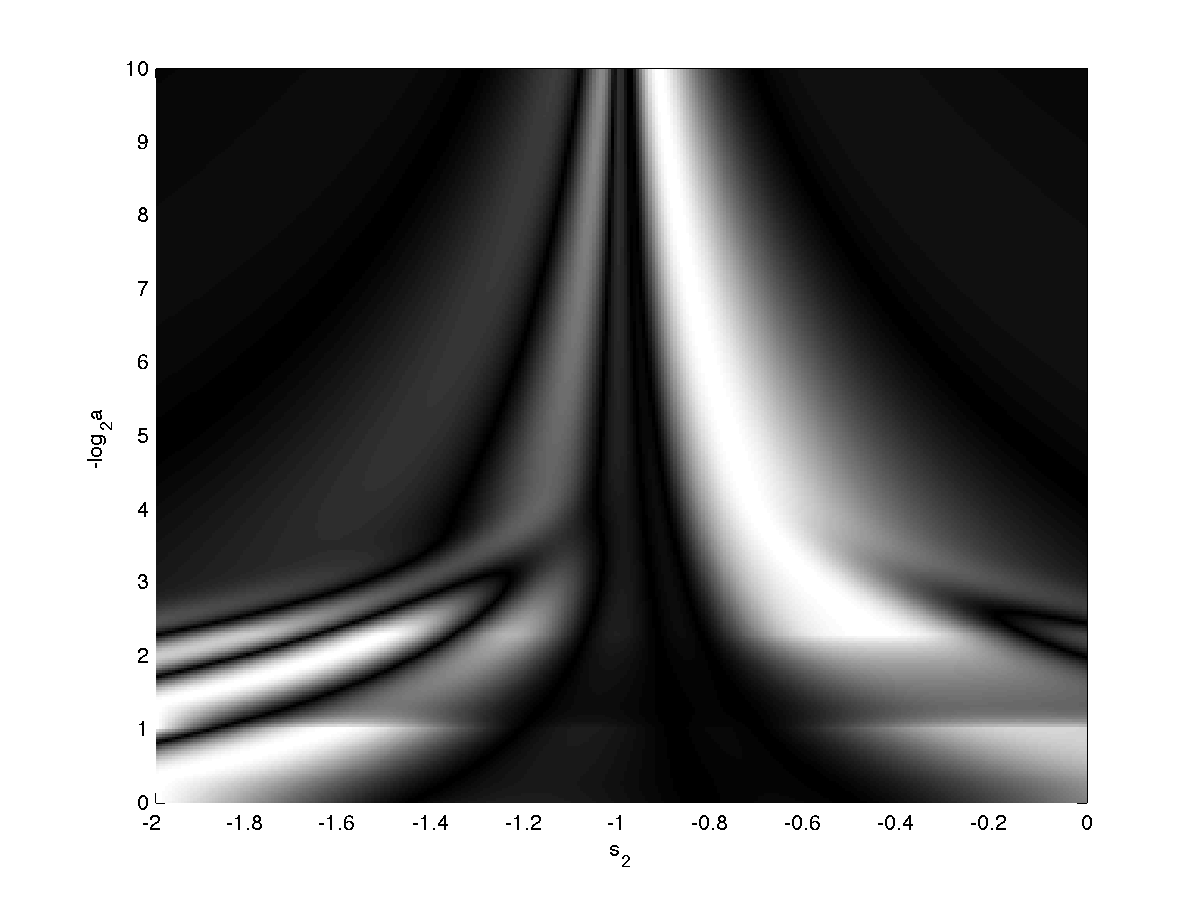}
\end{minipage}
\caption{Plots of the Taylorlet transform $\mathcal{T}f_i(a,s,0)$ for $f_1(x)=\1_{B_1}(x)$ (left column) and $f_2(x)=\1_{\R_+}(x_1-\cos x_2)$ (right column). The vertical axis shows the dilation parameter in a logarithmic scale $-\log_2a$. The horizontal axis shows the location $s_0$ (top), the slope $s_1$ (center) and the parabolic shear $s_2$ (bottom) rsp. The Taylorlet transform was computed for points $(a,s_i)$ on a $300\times 300-$grid. We can observe the paths of the local maxima \dtt wrt the respective shearing variable as they converge to the correct related geometric value through the scales. Due to the vanishing moment conditions of higher order, the local maxima display a fast convergence to the correct value. Consequently, we can obtain an approximation of $q_1(0)=-1,\dot q_1(0)=0,\ddot q_1(0)=1$ and $q_2(0)=1,\dot q_2(0)=0,\ddot q_2(0)=-1$ for the singularity functions $q_1(x_2)=-\sqrt{1-x_2^2}$ and $q_2(x_2)=\cos\(x_2\)$ of $f_1$ and $f_2$ rsp.}
\label{plot}
\end{figure}

Due to \cref{main}, the decay rate of the Taylorlet transform changes depending on the highest approximation order of the shearing variable. We can exploit this pattern in a step-by-step search for consecutive Taylor coefficients of the singularity function $q$. To this end, we first compute the Taylorlet transform of a function with varying shearing variable $s_0$ while $s_k=0$ for all $k\in\{1,\ldots,n\}$. Then we consider the propagation of the local maxima \dtt wrt $s_0$ through the scales. Since the choice $s_0=q(t)$ leads to the lowest decay rate, we can expect the local maxima near $s_0=q(t)$ to converge towards this value for decreasing scales in a similar fashion as in the method of wavelet maximum modulus by Mallat and Hwang \cite{mahw92}. Subsequently, we fix $s_0$ to the value $q(t)$ and search for the matching value of $s_1$ in the same way as in the preceding step for $s_0$. Due to the choice $s_0=q(t)$, the Taylorlet transform displays a low decay rate. Because of the vanishing moment conditions, the Taylorlet transform decays even slower, if additionally $s_1=\dot q(t)$. Hence, the method of modulus maximum is still applicable. With the same argumentation, we can repeat this procedure for all shearing variables up to the order of the Taylorlet. 

In order to better visualize the local maxima, we normalized the absolute value of the Taylorlet transform in the presented plots such that the maximum value in each scale is 1. Due to this normalization \dtt wrt the local maxima on a compact interval regarding the respective shearing variable (\dt eg the interval $[-1,1]\ni s_1$ in the center right image of \cref{plot}), discontinuities \dtt wrt the dilation parameter can appear. 

\section{Conclusion}

In order to detect higher order geometric information, we introduced the Taylorlet transform which is based on the continuous shearlet transform, and in addition to dilation, translation and classical shears utilizes shears of higher order. The transform allows for an extraction of position, orientation, curvature and other higher order geometric information of distributed singularities (\Cref{main}). A robust detection of these features can be guaranteed using the concept of vanishing moments of higher order. Additionally, we presented a constructive algorithm to build functions with the needed properties. First numerical studies showed its potential for future applications.

\section{Acknowledgements}

The author expresses his gratitude for the support by the DFG project FO 792/2-1 "Splines of complex order, fractional operators and applications in signal and image processing", awarded to Brigitte Forster.

\newpage

\bibliographystyle{alpha}
\bibliography{thesis}

\begin{thebibliography}{PKVDM97}

\bibitem[AA56]{AtAr56}
Fred Attneave and Malcolm~D Arnoult.
\newblock The quantitative study of shape and pattern perception.
\newblock {\em Psychological bulletin}, 53(6):452, 1956.

\bibitem[Apo76]{Ap76}
Tom~M Apostol.
\newblock {\em Introduction to analytic number theory}.
\newblock Springer-Verlag, 1976.

\bibitem[BO74]{blov74}
C~Blakemore and R~Over.
\newblock Curvature detectors in human vision?
\newblock {\em Perception}, 3(1):3--7, 1974.

\bibitem[CD05]{CaDo05a}
Emmanuel~J Candes and David~L Donoho.
\newblock Continuous curvelet transform: I. resolution of the wavefront set.
\newblock {\em Applied and Computational Harmonic Analysis}, 19(2):162--197,
  2005.

\bibitem[Gro11]{gr11}
Philipp Grohs.
\newblock Continuous shearlet frames and resolution of the wavefront set.
\newblock {\em Monatshefte f{\"u}r Mathematik}, 164(4):393--426, 2011.

\bibitem[KL09]{KuLa09}
Gitta Kutyniok and Demetrio Labate.
\newblock Resolution of the wavefront set using continuous shearlets.
\newblock {\em Transactions of the American Mathematical Society},
  361(5):2719--2754, 2009.

\bibitem[LPS16]{LePeSch16}
Christian Lessig, Philipp Petersen, and Martin Sch{\"a}fer.
\newblock Bendlets: A second-order shearlet transform with bent elements.
\newblock {\em arXiv preprint arXiv:1607.05520}, 2016.

\bibitem[MEO11]{MEO11}
Antonio Monroy, Angela Eigenstetter, and Bj{\"o}rn Ommer.
\newblock Beyond straight lines---object detection using curvature.
\newblock In {\em 2011 18th IEEE International Conference on Image Processing},
  pages 3561--3564. IEEE, 2011.

\bibitem[MH92]{mahw92}
Stephane Mallat and Wen~Liang Hwang.
\newblock Singularity detection and processing with wavelets.
\newblock {\em IEEE transactions on information theory}, 38(2):617--643, 1992.

\bibitem[PKVDM97]{PKM97}
Gabriele Peters, Norbert Kr{\"u}ger, and Christoph Von Der~Malsburg.
\newblock Learning object representations by clustering banana wavelet
  responses.
\newblock {\em Proceedings of the 1st STIPR}, pages 113--118, 1997.

\end{thebibliography}

\end{document}